\documentclass{article}
\usepackage[utf8]{inputenc}
\usepackage[margin=1 in]{geometry}
\usepackage{amsmath}
\usepackage{amsfonts}
\usepackage{amssymb}
\usepackage{amsthm}
\usepackage{mathrsfs}
\usepackage{tikz-cd}
\usepackage{enumitem}
\usepackage{lmodern}
\usepackage{hyperref}
\usepackage[vcentermath,enableskew]{youngtab}
\usepackage{ytableau}
\usepackage{mathtools}
\usepackage{graphbox}

\newtheorem{theorem}{Theorem}
\numberwithin{theorem}{section}
\newtheorem{lemma}[theorem]{Lemma}
\newtheorem{prop}[theorem]{Proposition}
\newtheorem{corollary}[theorem]{Corollary}

\newtheorem{claim}[theorem]{Claim}

\makeatletter
\let\c@theorem\c@figure
\makeatother

\newtheorem*{prop*}{Proposition}

\theoremstyle{definition}

\newtheorem{example}[theorem]{Example}

\newcommand{\mb}{\mathbb}
\newcommand{\tb}{\textbf}

\newcommand{\Z}{\mathbb{Z}}

\newcommand{\mc}{\mathcal}

\newcommand{\tn}{\textnormal}
\newcommand{\se}{\subseteq}

\newcommand{\til}{\widetilde}

\newcommand{\bs}{\backslash}

\newcommand{\ol}{\overline}
\newcommand{\lam}{\lambda}

\newcommand{\Sym}{\sf{Sym}}

\newcommand{\tcb}{\textcolor{blue}}

\def\multiset#1#2{\ensuremath{\left(\kern-.3em\left(\genfrac{}{}{0pt}{}{#1}{#2}\right)\kern-.3em\right)}}

\usepackage[maxbibnames=10,style=alphabetic,maxalphanames=10]{biblatex}
\title{Power sum expansions for Kromatic symmetric functions using Lyndon heaps}
\author{Laura Pierson \\ \href{mailto:lcpierson73@gmail.com}{lcpierson73@gmail.com}}

\begin{document}

\maketitle

\begin{abstract}
    Crew, Pechenik, and Spirkl \cite{crew2023kromatic} defined the \emph{\tb{\tcb{Kromatic symmetric function}}} $\ol{X}_G$ as a $K$-analogue of Stanley's chromatic symmetric function $X_G$ \cite{stanley1995symmetric}, and one question they asked was how $\ol{X}_G$ expands in their $\ol{p}_\lam$ basis, which they defined as a $K$-analogue of the classic \emph{\tb{\tcb{power sum basis}}} $p_\lam.$ We gave a formula in \cite{pierson2024power} that partially answered this question but did not explain the combinatorial significance of the coefficients. Here, we give combinatorial descriptions for the $\ol{p}$-coefficients of $\ol{X}_G$ and $\omega(\ol{X}_G)$, lifting the $p$-expansion of $X_G$ in terms of \emph{\tb{\tcb{acyclic orientations}}} that was given by Bernardi and Nadeau in \cite{bernardi2020combinatorial}. We also propose an alternative $K$-analogue $\ol{p}'$ of the $p$-basis that gives slightly cleaner expansion formulas. Our expansions are based on \emph{\tb{\tcb{Lyndon heaps}}}, introduced by Lalonde \cite{lalonde1995lyndon}, which are representatives for certain equivalence classes of acyclic orientations on clan graphs of $G$. Additionally, we show that knowing $\ol{X}_G$ is equivalent to knowing the multiset of independence polynomials of induced subgraphs of $G$, which gives shorter proofs of all our results from \cite{pierson2024counting} that $\ol{X}_G$ can be used to determine the number of copies in $G$ of certain induced subgraphs. We also give power sum expansions for the \emph{\tb{\tcb{Kromatic quasisymmetric function}}} $\ol{X}_G(q)$ defined by Marberg in \cite{marberg2023kromatic} in the case where $G$ is the incomparability graph of a unit interval order.
\end{abstract}

\section{Introduction}\label{sec:intro}

Let $G$ be a vertex-weighted graph with vertex set $V$, edge set $E$, and weight function $w:V\to \Z_{>0}.$ The \emph{\tb{\tcb{chromatic symmetric function}}} is $$X_{(G,w)}(x_1,x_2,\dots) \coloneqq \sum_{\kappa\tn{ a proper coloring}}\ \prod_{v\in V}x_{\kappa(v)}^{w(v)},$$ where a \emph{\tb{\tcb{proper coloring}}} $\kappa:V\to \Z_{>0}$ assigns a positive integer color to each vertex such that adjacent vertices get different colors. $X_G$ was introduced by Stanley in \cite{stanley1995symmetric} as a symmetric function analogue for the chromatic polynomial $\chi_G(q),$ and the vertex-weighted version was introduced by Crew and Spirkl in \cite{crew2020deletion}. The \emph{\tb{\tcb{Kromatic symmetric function}}} is $$\ol{X}_{(G,w)}(x_1,x_2,\dots) \coloneqq \sum_{\kappa\tn{ a proper set coloring}}\ \prod_{v\in V} \prod_{i\in \kappa(v)}x_i^{w(v)},$$ where a \emph{\tb{\tcb{proper set coloring}}} $\kappa:V\to 2^{\Z_{>0}}\bs\{\varnothing\}$ assigns a nonempty set of colors to each vertex so adjacent vertices get nonoverlapping color sets. Our focus here will be on the unweighted case where all vertices have weight 1, so we will write simply $X_G$ and $\ol{X}_G$, but we mention the weighted versions to help motivate the \emph{\tb{\tcb{power sum symmetric functions}}} that we will be using.

$\ol{X}_G$ was introduced by Crew, Pechenik, and Spirkl in \cite{crew2023kromatic} as a $K$-analogue of $X_G$ based on the ideas of \emph{\tb{\tcb{combinatorial $K$-theory}}} \cite{buch2005combinatorial}, which involves deforming the cohomology ring of a variety by introducing an extra parameter $\beta$ (usually set to $-1$) to get a new ring structure called the $K$-ring. The authors of \cite{crew2023kromatic} defined $\ol{X}_G$ in the hopes of shedding light on possible geometric interpretations of $X_G$, especially in relation to Gasharov's Schur-positivity result for incomparability graphs of (3+1)-free posets \cite{gasharov1996incomparability}, and the Stanley--Stembridge conjecture that $X_G$ is $e$-positive for the same class of graphs. The Stanley--Stembridge conjecture has since been proven by Hikita \cite{hikita2024proof}, but its geometric meaning is still not fully understood, although Kato \cite{kato2024geometric} recently gave a new geometric interpretation for $X_G$ along with some insight into how it may be related to the Stanley--Stembridge conjecture and its graded version, the Shareshian--Wachs conjecture. 

In addition to the Schur and $e$-bases, $X_G$ also has an interesting expansion in the \emph{\tb{\tcb{power sum basis}}} $p_\lam.$ In Stanley's original paper \cite{stanley1995symmetric}, he gave a $p$-expansion for $X_G$ as a sum over subsets of the edges, $X_G = \sum_{S\se E} (-1)^{|S|}p_{\lam(S)},$ where the parts of $\lam(S)$ are the sizes of the connected components. He also gave modifications of that expansion in terms of the broken circuit complex and in terms of the M\"obius function of the poset of connected subgraphs of $G$. Of more interest to us here, however, is a newer $p$-expansion in terms of acyclic orientations due to Bernardi and Nadeau \cite{bernardi2020combinatorial}, proven using the theory of heaps:

\begin{theorem}[\cite{bernardi2020combinatorial}, Proposition 5.3]\label{thm:bernadi_nadeau}
    $(-1)^{|\lam|-\ell(\lam)}[p_\lam]X_G$ counts acyclic orientations of $G$ such that the sizes of the source components are the parts of $\lam$ in some order.
\end{theorem}

The notation $[p_\lam]X_G$ denotes the coefficient of $p_\lam$ in the $p$-expansion of $X_G$. The \emph{\tb{\tcb{source components}}} of an acyclic orientation are defined as follows. Fix a total ordering on $V$. The first source component contains the first vertex and all vertices reachable from it by directed paths. After that, each new source component consists of the first vertex not yet visited, and all vertices reachable from it and not yet visited.

A natural question about $\ol{X}_G$ is thus whether there is a $K$-lift of one of these $p$-expansion formulas. The authors of \cite{crew2023kromatic} proposed a $K$-theoretic $\ol{p}$-basis which can be defined by 
\begin{equation}
    1 + \ol{p}_k(x_1,x_2,\dots) \coloneqq \prod_{i\ge 1} (1 + x_i^k), \hspace{1cm}\ol{p}_\lam \coloneqq \ol{p}_{\lam_1}\dots\ol{p}_{\ell(\lam)}.
\end{equation}
As noted in \cite{crew2023kromatic}, we can think of $\ol{p}_k$ as the Kromatic symmetric function for a single weight $k$ vertex, and $\ol{p}_\lam$ as the Kromatic symmetric function for an edgeless graph with vertex weights $\lam_1,\dots,\lam_{\ell(\lam)}.$ One can similarly think of the classic power sum function $p_\lam$ as the chromatic symmetric function for an edgeless graph with vertex weights $\lam_1,\dots,\lam_{\ell(\lam)}.$

The authors of \cite{crew2023kromatic} computed the leading terms of the $\ol{p}$-expansions of $\ol{X}_G$ numerically for some small graphs $G$, saw that the coefficients appeared to be integers, and asked what those integers meant. In \cite{pierson2024power}, we showed that the coefficients are in fact integers and gave an explicit formula for them, but we did not give an answer to what those integers count. We answer that question here using the theory of heaps.

\subsection{Power sum expansions for \texorpdfstring{$\ol{X}_G$}{XG}}

Fix a total order on the vertex set $V$ of $G$. For a composition $\alpha = (\alpha_1,\dots,\alpha_{|V|})$ with $\alpha_1,\dots,\alpha_{|V|} \in \mb{Z}_{\ge 0},$ the \emph{\tb{\tcb{$\boldsymbol{\alpha}$-clan graph}}} $\textsf{Clan}_\alpha(G)$ is formed by blowing up the $i$th vertex of $G$ into $\alpha_i$ vertices for every $i,$ such that vertices are adjacent in $\textsf{Clan}_\alpha(G)$ if they come from either the same vertex or adjacent vertices in $G$. (This differs from the definition in \cite{crew2023kromatic} because we allow for parts of $\alpha$ to be 0.) A \emph{\tb{\tcb{heap}}} $H$ of type $\alpha$ on $G$ is an acyclic orientation on $\textsf{Clan}_\alpha$ and its \emph{\tb{\tcb{size}}} is $|H| \coloneqq |\alpha|=\alpha_1+\dots+\alpha_{|V|}$. Each copy of a vertex in a heap is called a \emph{\tb{\tcb{piece}}}, and pieces corresponding to the same vertex are considered indistinguishable (i.e. swapping two copies of the same vertex gives the same heap.) Heaps were introduced by Viennot in \cite{viennot2006heaps}, based on ideas of Cartier and Foata from \cite{cartier2006problemes}. Lalonde \cite{lalonde1995lyndon} defined \emph{\tb{\tcb{Lyndon heaps}}} as canonical representatives for certain equivalence classes of heaps under an equivalence relation analogous to cyclic shifts of words. Lyndon heaps are analogous to \emph{\tb{\tcb{Lyndon words}}}, which are the lexicographically minimal representatives of equivalence classes of aperiodic words under cyclic shifts.

Lyndon heaps are also closely related to Lie algebras. Specifically, the \emph{\tb{\tcb{free partially commutative Lie algebra}}} on $G$ consists of all combinations of vertices formed by repeatedly taking linear combinations or applying the bracket relation, subject to the relation that nonadjacent vertices commute. The subspace of this Lie algebra spanned by elements formed by repeatedly applying the bracket in a way that uses vertex $i$ $\alpha_i$ times for each $i$ (called the \emph{\tb{\tcb{$\boldsymbol{\alpha}$-grade space}}}) has a basis whose elements correspond to Lyndon heaps on $G$ of type $\alpha.$ The relationship between Lyndon heaps, Lie algebra dimensions, and coefficients of the chromatic polynomials for clan graphs was studied by Arunkumar in \cite{arunkumar2022generalised}, who also draws a connection to Kac--Moody Lie algebras.

Using heaps, we lift Theorem \ref{thm:bernadi_nadeau} to give a combinatorial interpretation for the $\ol{p}$-coefficients of $\ol{X}_G$:

\begin{theorem}\label{thm:XG_old_ps}
    $(-1)^{|\lam|-\ell(\lam)}[\ol{p}_\lam]\ol{X}_G$ counts multisets $\{L_1,\dots,L_{\ell(\lam)}\}$ of Lyndon heaps on $G$ such that each vertex is used by at least one $L_i$, $L_i$ either has size $\lam_i$ or has some even size $\lam_i/2^j$ with $j\ge 1$, and the Lyndon heaps corresponding to repeated odd parts of $\lam$ are all distinct.
\end{theorem}

This verifies our statement from \cite[Theorem 1.6]{pierson2024power} about the signs of the $\ol{p}$-coefficients, and also gives a combinatorial explanation for why they are integers. We can roughly visualize the Lyndon heap coverings here and in our other results as being similar to the source component partitions in Theorem \ref{thm:bernadi_nadeau}, except that the Lyndon heaps can overlap while the source components could not, and vertices are allowed to get used multiple times within the same Lyndon heap, source components cannot have repeated vertices.

We can recover Theorem \ref{thm:bernadi_nadeau} from Theorem \ref{thm:XG_old_ps}, since for $\lam \vdash |V|,$ the $p$-coefficients of $X_G$ need to match the $\ol{p}$-coefficients of $\ol{X}_G,$ because $X_G$ consists of the lowest degree terms of $\ol{X}_G$ and $p_\lam$ consists of the lowest degree terms of $\ol{p}_\lam$. For $\lam \vdash |V|,$ the only way that all vertices get used is if each vertex is used in exactly one Lyndon heap, and the heap sizes actually match the parts $\lam_i$ of $\lam$ (since if a heap had size $\lam_i/2^j$ for $j>1,$ the total number of vertices used by all the heaps would be less than $|\lam|=|V|$, meaning not all vertices would get used). Lyndon heaps with no repeat vertices are the same as acyclic orientations where the smallest vertex is the unique source. A collection of such acyclic orientations on subsets of the vertices can be turned into an acyclic orientation on all of $G$ by ordering the source components according to their sources, then directing all remaining edges from later to earlier source components. That gives an acyclic orientation where the sizes of the source component match the parts of $\lam,$ as in Theorem \ref{thm:bernadi_nadeau}.

We can also make Theorem \ref{thm:XG_old_ps} (and our other results below) look more similar to Theorem \ref{thm:bernadi_nadeau} as follows. Heaps have a multiplicative structure given by composition, and every heap has a unique \emph{\tb{\tcb{Lyndon factorization}}} into Lyndon heaps $L_1\circ \dots \circ L_k$ with $L_1\ge \dots \ge L_k$, where the ordering on heaps is defined by identifying heaps with words in a certain way and ordering the words lexicographically. Thus, given a multiset of heaps as in Theorem \ref{thm:XG_old_ps}, we can get a corresponding single heap on $G$ of size $|\lam|$ by replicating each heap of size $\lam_i/2^j$ a total of $2^j$ times, and then composing of all the resulting heaps in reverse lexicographic order. The Lyndon factors of this heap then play the same role as the source components of the acyclic orientation in Theorem \ref{thm:bernadi_nadeau}.

We can also ask about the $\ol{p}$-expansion of $\omega(\ol{X}_G)$. The classic involution $\omega$ on the ring of symmetric functions satisfies $\omega(p_\lam) = (-1)^{|\lam|-\ell(\lam)}p_\lam,$ so Theorem \ref{thm:bernadi_nadeau} immediately gives the $p$-expansion of $\omega(X_G)$, and shows that $\omega(X_G)$ is $p$-positive. For the $\ol{p}$-basis, however, $\omega(\ol{p}_\lam) \ne (-1)^{|\lam|-\ell(\lam)}\ol{p}_\lam,$ so the $\ol{p}$-expansion of $\omega(\ol{X}_G)$ is slightly different, but $\omega(\ol{X}_G)$ still turns out to be $\ol{p}$-positive:

\begin{theorem}\label{thm:omega_XG_old_ps}
    $[\ol{p}_\lam]\omega(\ol{X}_G)$ counts multisets $\{L_1,\dots,L_{\ell(\lam)}\}$ of Lyndon heaps on $G$ such that each vertex is used by at least one $L_i$, $L_i$ has size $\lam_i/2^j$ for some $j\ge 0,$ and the Lyndon heaps corresponding to the same repeated part of $\lam$ are distinct from each other. (The same $L_i$ may occur multiple times if its repeated occurrences correspond to parts of $\lam$ of different sizes.)
\end{theorem}

Alternatively, we can get slightly nicer expansions by modifying our $\ol{p}$-basis. Define the $\ol{p}'$ basis by \begin{equation}\label{eqn:new_p_def}
    1+\ol{p}'_k(x_1,x_2,\dots) \coloneqq \prod_{i\ge 1}\frac1{1-x_i^k} = \prod_{i\ge 1}(1 + x_i^k + x_i^{2k} + \dots), \hspace{1cm} \ol{p}'_\lam \coloneqq \ol{p}'_{\lam_1}\dots\ol{p}'_{\lam_{\ell(\lam)}}.
\end{equation} We can think of $\ol{p}'_k$ as the generating series for colorings of a weight $k$ vertex where each color is allowed to be used any number of times as long as the vertex gets at least one color.

Now we get:

\begin{theorem}\label{thm:XG_new_ps}
    $(-1)^{|\lam|-\ell(\lam)}[\ol{p}_\lam']\ol{X}_G$ counts multisets $\{L_1,\dots,L_{\ell(\lam)}\}$ of Lyndon heaps on $G$ such that each vertex is used by at least one $L_i,$ $L_i$ can have size $\lam_i$ or $\lam_i/2$ if $\lam_i/2$ is odd, and the Lyndon heaps corresponding to repeated odd parts of $\lam$ are distinct.
\end{theorem}

We get a cleaner expansion for $\omega(\ol{X}_G)$, which is $\ol{p}'$-positive:

\begin{theorem}\label{thm:omega_XG_new_ps}
    $[\ol{p}_\lam']\omega(\ol{X}_G)$ counts sets $\{L_1,\dots,L_{\ell(\lam)}\}$ of distinct Lyndon heaps on $G$ such that each vertex is used by at least one $L_i,$ and $L_i$ has size $\lam_i.$ 
\end{theorem}

Equivalently, $[\ol{p}_\lam']\omega(\ol{X}_G)$ counts heaps on $G$ such that the factors in its Lyndon factorization are all distinct, and their sizes are the parts of $\lam$ in some order.

\subsection{Counting induced subgraphs using \texorpdfstring{$\ol{X}_G$}{XG}}

There exist nonisomorphic graphs with the same chromatic symmetric function \cite{orellana2014graphs,stanley1995symmetric}. However, an extensively studied but still open question about $X_G$ is whether it is always different for different trees \cite{aliste2024chromatic, crew2022note, heil2018algorithm, martin2008distinguishing, wang2024class}. $\ol{X}_G$ contains more information than $X_G$ and is known to distinguish some graphs that $X_G$ cannot \cite{crew2023kromatic}, but it is unknown whether it distinguishes all graphs or even all trees. In \cite{pierson2024counting}, we conjectured that $\ol{X}_G$ does distinguish all graphs and gave evidence toward that conjecture by showing that the number of copies in $G$ of certain induced subgraphs can be recovered from $\ol{X}_G$, by setting up large systems of linear equations based on the $\ol{\til{m}}$-expansion of $\ol{X}_G$ from \cite{crew2023kromatic}. As a corollary of our proofs here, we show the following:

\begin{corollary}\label{cor:subgraphs}
    Knowing the Kromatic symmetric function $\ol{X}_G$ of a graph is equivalent to knowing the multiset of independence polynomials of induced subgraphs of $G$.
\end{corollary}

In particular, for any $H$ such that no nonisomorphic graph has the same independence polynomial as $H$, the number of induced copies of $H$ in $G$ can be determined from $\ol{X}_G$. This includes all the induced subgraphs we showed could be counted in \cite{pierson2024counting} (seven graphs on 4 vertices, eleven graphs on 5 vertices, and all graphs that are a union of a star and some isolated vertices). It also includes cliques, cliques minus an edge, paths of length 2 and all odd lengths (based on \cite{beaton2018independence}), and \emph{\tb{\tcb{threshold graphs}}} (based on \cite{stevanovic1997clique}), which are graphs with no induced $P_4$, $C_4$, or $\ol{C_4}$, or equivalently, graphs formed by repeatedly adding an isolated vertex or a vertex connected to all previous vertices. For $T$ a tree, \cite{levit2008roots} implies that $\ol{X}_T$ also counts induced subgraphs isomorphic to each \emph{\tb{\tcb{well-covered spider}}}, i.e. induced subgraphs with at most one vertex of degree at least three and all maximal independent sets the same size.

\subsection{Quasisymmetric power sum expansions}

Again, fix a total ordering on $V$. Shareshian and Wachs \cite{shareshian2016chromatic} defined the \emph{\tb{\tcb{chromatic quasisymmetric function}}} as $$X_G(q) \coloneqq \sum_{\kappa\tn{ a proper coloring}} q^{\tn{asc}(\kappa)}\prod_{v\in V}x_{\kappa(v)},$$ where $\tn{asc}(\kappa)$ counts \emph{\tb{\tcb{ascents}}} of $\kappa,$ i.e. pairs of vertices $u,v\in V$ with $u < v$ and $\kappa(u) < \kappa(v).$ $X_G(q)$ is not symmetric in general, but it is in the case of \emph{\tb{\tcb{incomparability graphs of unit interval orders}}}, which are a subset of the graphs for which the Stanley--Stembridge conjecture applies, and also are precisely the graphs for which Kato's geometric realization of chromatic quasisymmetric functions in \cite{kato2024geometric} applies. For this unit interval case, Athanasiadis \cite{athanasiadis2014power} gave the following $p$-expansion formula for $X_G(q)$:

\begin{theorem}[{Athanasiadis \cite[Corollary 7]{athanasiadis2014power}}]\label{thm:athanasiadis}
    If $G$ is the incomparability graph of a unit interval order, $$\omega(X_G(q)) = \sum_{\lam \vdash |V|} \frac{p_\lam}{z_\lam} \sum_{\pi\in \Pi_\lam}\  \sum_{o\in \mc{AO}^*(G,\pi)}q^{\tn{asc}(o)}.$$
\end{theorem}

In this statement, $\Pi_\lam$ is the set of \emph{\tb{\tcb{block decompositions}}} of $V$ of type $\lam,$ meaning partitions of $V = B_1 \sqcup \dots \sqcup B_{\ell(\lam)}$ with $|B_i| = \lam_i,$ and $\mc{AO}^*(G,\pi)$ is the set of acyclic orientations of $G$ such that the restriction to each $B_i$ has a unique source and for $i<j$, all edges between $B_i$ and $B_j$ are directed towards $B_j.$ (Athanasiadis actually stated this using sinks instead of sources, but we use sources here for consistency with \cite{bernardi2020combinatorial} and our other results here. It can be shown to hold for sources as well by switching the directions of all edges in each acyclic orientation, which replaces each $q^k$ term with a $q^{|E|-k}$ term since each ascent becomes a non-ascent and vice versa, and then by using the fact from \cite{shareshian2016chromatic} that the coefficient of each $p_\lam/z_\lam$ is a palindromic polynomial in $q.$)

Marberg \cite{marberg2023kromatic} defined two different quasisymmetric analogues for $\ol{X}_G$. We will focus here on his second version, $$\ol{X}_G(q) \coloneqq \sum_{\kappa\tn{ a proper set coloring}} q^{\tn{asc}(\kappa)} \prod_{v\in V}\prod_{i\in \kappa(v)} x_i,$$ where $\tn{asc}(\kappa)$ counts quadruples $(u,v,i,j)$ with $u,v\in V,$ $u<v,$ $i\in \kappa(u), j\in \kappa(v)$, and $i<j.$ Marberg shows that $\ol{X}_G(q)$, like $X_G(q)$, is symmetric when $G$ is the incomparability graph of a unit interval order. This means that in those cases it must have an expansion in the $\ol{p},$ $\omega(\ol{p})$, $\ol{p}',$ and $\omega(\ol{p}')$ bases, and those expansions must specialize to Theorems \ref{thm:XG_old_ps}, \ref{thm:omega_XG_old_ps}, \ref{thm:XG_new_ps}, and \ref{thm:omega_XG_new_ps} when $q=1,$ since $\ol{X}_G(1) = \ol{X}_G.$ Thus, a natural question is what those expansions are. They are not especially simple, but we can still give explicit formulas for them, which we derive in \S \ref{sec:quasi} by finding transition formulas between the ordinary $p$-basis and the $K$-theoretic $p$-bases, and then combining those transition formulas with Theorem \ref{thm:athanasiadis}. Specializing those expansions to $q=1$ gives us an alternative proof for Theorems \ref{thm:XG_old_ps}, \ref{thm:omega_XG_old_ps}, \ref{thm:XG_new_ps}, and \ref{thm:omega_XG_new_ps}.

\bigskip

Another question one might ask is whether there is some sort of quasisymmetric power sum expansion for the other Kromatic quasisymmetric function $\ol{L}_G(q)$ that Marberg defined in \cite{marberg2023kromatic}. Unlike $\ol{X}_G(q)$, $\ol{L}_G(q)$ is symmetric in very few cases, which means we cannot expect it to have expansions in our $\ol{p},$ $\omega(\ol{p})$, $\ol{p}',$ or $\omega(\ol{p}')$ bases, since they are all symmetric. However, one might hope that it has an expansion in some quasisymmetric $K$-analogue $\ol{\Psi}_\alpha$ of the quasisymmetric power sum basis $\Psi_\alpha$ studied by Alexandersson and Sulzgruber in \cite{alexandersson2021p}, ideally lifting the expansion formula from \cite{alexandersson2021p} for $X_G(q)$ in the $\Psi_\alpha$ basis. From our attempts to find such a quasisymmetric power sum expansion for $\ol{L}_G(q),$ however, we think that the $\psi_\alpha$-expansion of $X_G(q)$ from \cite{alexandersson2021p} may not have such a $K$-analogue.

\subsection{Organization}

The remainder of this paper is organized as follows:
\begin{itemize}
    \item In \S \ref{sec:background}, we introduce relevant background on graphs, symmetric functions, and heaps.
    \item In \S \ref{sec:dirichlet}, we derive Theorems \ref{thm:XG_old_ps}, \ref{thm:omega_XG_old_ps}, \ref{thm:XG_new_ps}, and \ref{thm:omega_XG_new_ps} using M\"obius inversion and properties of heap generating series. We give the proofs in the reverse order of the statements above (i.e. Theorem \ref{thm:omega_XG_new_ps}, then \ref{thm:XG_new_ps}, then \ref{thm:omega_XG_old_ps}, then \ref{thm:XG_old_ps}) so that the easier proofs are first. Then we prove Corollary \ref{cor:subgraphs} in \S \ref{sec:subgraphs}.
    \item In \S\ref{sec:direct}, we give simpler bijective explanations for the key factorizations that are used in the proofs of our expansion formulas in \S \ref{sec:dirichlet}. Again, we give the proofs in the order of Theorem \ref{thm:omega_XG_new_ps}, then \ref{thm:XG_new_ps}, then \ref{thm:omega_XG_old_ps}, then \ref{thm:XG_old_ps} so that the simpler proofs are first.
    \item Finally, in \S\ref{sec:quasi}, we derive quasisymmetric power sum expansion formulas for $\ol{X}_G(q)$ in the case where $G$ is the incomparability graph of a unit interval order, and we show how they specialize to Theorems \ref{thm:XG_old_ps}, \ref{thm:omega_XG_old_ps}, \ref{thm:XG_new_ps}, and \ref{thm:omega_XG_new_ps} at $q=1.$
\end{itemize} 

\section{Background}\label{sec:background}

\subsection{Graphs}

A \emph{\tb{\tcb{graph}}} $G = (V,E)$ is a set $V$ of vertices and a set $E$ of \emph{\tb{\tcb{edges}}}, where each edge is an unordered pair of distinct vertices. We will assume our graphs are unweighted and have no multiple edges. Two vertices are \emph{\tb{\tcb{adjacent}}} if there is an edge between them, and an \emph{\tb{\tcb{independent set}}} is a subset of $V$ containing no two adjacent vertices. For a subset $W\se V,$ the \emph{\tb{\tcb{induced subgraph}}} $G|_W$ is the graph with vertex set $W$ and edges $\{vw:v,w\in W\tn{ and }vw \in E\}.$ Two graphs $G = (V,E)$ and $G' = (V',E')$ are \emph{\tb{\tcb{isomorphic}}} if there is a bijection $\phi:V\to V'$ such that $vw\in E$ if and only if $\phi(v)\phi(w)\in E'.$

An \emph{\tb{\tcb{acyclic orientation}}} assigns a direction $v\to w$ to each edge $vw$ such that there are no directed cycles $v_1\to v_2\to \dots \to v_k \to v_1$ where each $v_i\to v_{i+1}$ is a directed edge. A \emph{\tb{\tcb{source}}} of an acyclic orientation is a vertex with no edges directed towards it.

A \emph{\tb{\tcb{partially ordered set (poset)}}} $P$ is a set together with an ordering relation $\le$ applying to some ordered pairs of elements such that $\le$ is reflexive, antisymmetric and transitive: that is, $v \le v$ for all $v$, the only way that $v\le w$ and $w \le v$ is if $v=w,$ and if both $u \le v$ and $v\le w$ then $u\le w.$ We write $v<w$ if $v\le w$ and $v\ne w.$ Two elements $v,w\in P$ are \emph{\tb{\tcb{incomparable}}} if $v\not \le w$ and $w\not \le v.$ The \emph{\tb{\tcb{incomparability graph}}} $G$ of $P$ is the graph with vertex set $P$ and edges $vw \in E$ if and only if $v$ and $w$ are incomparable in $P$. For $a,b\in \Z_{>0},$ we say that $P$ is \emph{\tb{\tcb{$\boldsymbol{(a+b)}$-free}}} if there are no chains $v_1< \dots < v_a$ and $w_1 < \dots < w_b$ in $P$ with every $v_i$ incomparable to every $w_j.$ A \emph{\tb{\tcb{unit interval order}}} is a poset that is both $(2+2)$-free and $(3+1)$-free. Equivalently, $P$ is a unit interval order if and only if there is a map $\phi:P\to \mathbb{R}$ such that $v < w$ in $P$ if and only if $\phi(v) + 1 < \phi(w)$, so $v$ and $w$ are incomparable if and only if $\phi(v)$ and $\phi(w)$ are within 1 unit of each other.

\subsection{Symmetric functions}

A \emph{\tb{\tcb{symmetric function}}} is a power series in a countable set of variables $x_1,x_2,\dots$ that is invariant under any permutation of the variables. The symmetric functions form a ring $\Sym$, and bases for this ring are generally indexed by partitions, where a \emph{\tb{\tcb{partition}}} $\lam$ is a list $(\lam_1,\dots,\lam_{\ell(\lam)})$ of positive integers with $\lam_1\ge \dots \ge \lam_{\ell(\lam)}.$ The $\lam_i$'s are the \emph{\tb{\tcb{parts}}} of $\lam,$ $\ell(\lam)$ is the \emph{\tb{\tcb{length}}} or number of parts, and $|\lam| \coloneqq \lam_1+\dots+\lam_{\ell(\lam)}$ is the \emph{\tb{\tcb{size}}}. We write $i_k(\lam)$ to denote the \emph{\tb{\tcb{multiplicity}}} of $k$ as a part of $\lam,$ i.e. the number of parts of $\lam$ that are equal to $k,$ and we define $$z_\lam \coloneqq \prod_{k\ge 1} k^{i_k(\lam)}\cdot i_k(\lam)!.$$ The main basis for $\Sym$ that we will be interested in here is the \emph{\tb{\tcb{power sum symmetric functions}}} $$p_n(x_2,x_2,\dots) \coloneqq x_1^n+x_2^n+\dots,\hspace{1cm}p_\lam \coloneqq p_{\lam_1}\dots p_{\lam_{\ell(\lam)}}.$$ We will also use two other bases in some of our proofs. The \emph{\tb{\tcb{elementary symmetric functions}}} are $$e_n(x_1,x_2,\dots) \coloneqq \sum_{i_1< \dots <i_n} x_{i_1}\dots x_{i_n},\hspace{1cm}e_\lam \coloneqq e_{\lam_1}\dots e_{\lam_{\ell(\lam)}},$$ so $e_n$ is the sum of all products of $n$ distinct variables. The \emph{\tb{\tcb{homogeneous symmetric functions}}} are $$h_n(x_1,x_2,\dots) \coloneqq \sum_{i_1\le \dots \le i_n} x_{i_1}\dots x_{i_n},\hspace{1cm}h_\lam \coloneqq h_{\lam_1}\dots h_{\lam_{\ell(\lam)}},$$ so $h_n$ is the sum of all products of $n$ variables with repetition allowed. The \emph{\tb{\tcb{involution $\boldsymbol{\omega}$}}} on $\Sym$ is a ring homomorphism satisfying $\omega(p_\lam) = (-1)^{|\lam|-\ell(\lam)}p_\lam$ and $\omega(h_\lam) = e_\lam.$

A \emph{\tb{\tcb{quasisymmetric function}}} is a power series in $x_1,x_2,\dots$ such that for any fixed composition $\alpha = (\alpha_1,\dots,\alpha_n),$ all monomials of the form $x_{i_1}^{\alpha_1}\dots x_{i_n}^{\alpha_n}$ have the same coefficient as long as $i_1 < \dots < i_n$. All symmetric functions are quasisymmetric, and much of the combinatorics of symmetric functions can be extended to quasisymmetric functions, but we will not go into that here since the quasisymmetric functions we will be interested in will actually be symmetric.

The chromatic symmetric function $X_G,$ the Kromatic symmetric function $\ol{X}_G$, the Kromatic quasisymmetric function $\ol{X}_G(q),$ and our two $K$-analogues $\ol{p}_\lam$ and $\ol{p}'_\lam$ for the $p$-basis were already defined in \S \ref{sec:intro}, so we do not repeat their definitions here.

\subsection{Heaps}

We defined a \emph{\tb{\tcb{heap}}} $H$ in \S \ref{sec:intro} as an acyclic orientation of an $\alpha$-clan graph of $G$. The copies of each vertex are called \emph{\tb{\tcb{pieces}}} and the \emph{\tb{\tcb{size}}} $|H|$ is the number of pieces. We write $\mc{H}_G$ for the set of all heaps on $G$, and $\mc{H}_G(n)$ for the set of heaps of size $n$ on $G$.

A heap of type $\alpha$ can also be thought of as the equivalence class of all words with alphabet $V$ where vertex $i$ shows up $\alpha_i$ times for every $i$ and $v$ comes before $w$ in the word whenever $v\to w$ is a directed edge in the acyclic orientation. The equivalence relation is thus that nonadjacent vertices can always be moved past each other in the word. These equivalence classes of words over all choices of $\alpha$ form a \emph{\tb{\tcb{free partially commutative monoid}}} where the operation is composition of words, i.e. $H_1\circ H_2$ is the equivalence class of words that can be written as a word for $H_1$ followed by a word for $H_2.$ The monoid is partially commutative because two letters commute if and only if the corresponding vertices are nonadjacent in $G$. A heap is \emph{\tb{\tcb{connected}}} if the corresponding $\alpha$-clan graph is connected. A heap is \emph{\tb{\tcb{periodic}}} with period $k$ if it can be written as the $k$-fold composition $H^k = H\circ \dots \circ H$ of a heap with itself for any $k>1,$ and the heap is \emph{\tb{\tcb{aperiodic}}} otherwise.

A \emph{\tb{\tcb{pyramid}}} is a heap where the corresponding acyclic orientation has a unique source. We write $\mc{P}_G$ for the set of pyramids on $G$, and $\mc{P}_G(n)$ for the set of pyramids of size $n$ on $G$.

A \emph{\tb{\tcb{rotation}}} operation on pyramids can be defined as follows. For any piece $p\in H,$ the \emph{\tb{\tcb{pyramid determined by $\boldsymbol{p}$}}} in $H$ is the set of vertices reachable from $p$ by a directed path. For any $p\in H,$ $H$ can be written in the form $U\circ V$ where $V$ is the pyramid of $p.$ Lalonde \cite{lalonde1995lyndon} defined $A_p(H) \coloneqq V\circ U$ as the heap formed by interchanging $U$ and $V$. The new heap $U\circ V$ may no longer be a pyramid, but he shows that if this operation is repeated with the same $p,$ then $A_p^n(H)$ eventually stabilizes to a fixed pyramid $A_p^\infty(H),$ which is the rotation of the original pyramid that has $p$ as the new source. Furthermore, if a pyramid starts with $p_1$ as the source and is rotated to have $p_2$ as the source, and then rotated again to have $p_3$ as the source, the result is the same as if the original pyramid was directly rotated to have $p_3$ as the source. Thus, this rotation operation defines an \emph{\tb{\tcb{equivalence relation on pyramids}}}.

All pyramids are connected since there must be a path from the source to all other vertices. Lalonde \cite{lalonde1995lyndon} shows that within each equivalence class, either the pyramids for different pieces are all distinct (including for different pieces corresponding to the same vertex) and aperiodic, or else all pyramids in the equivalence class are periodic with the same period. This fact will be important for us later.

Given an ordering on the set of vertices, the \emph{\tb{\tcb{standard word}}} $\tn{St}(H)$ associated to a heap is the lexicographically maximal word in its equivalence class. There is then a total ordering on heaps given by $H_1 \ge H_2$ if $\tn{St}(H_1)\ge \tn{St}(H_2).$ Within each equivalence class of aperiodic pyramids defined above, the \emph{\tb{\tcb{Lyndon heap}}} is the minimal pyramid in the equivalence class under this ordering. We write $\mc{L}_G$ for the set of Lyndon heaps on $G$, and $\mc{L}_G(n)$ for the set of Lyndon heaps of size $n$ on $G$.

We give an example now to help illustrate these definitions, and in particular how to find the equivalence class and the Lyndon heap for a given pyramid:

\begin{example}
    Let $G = P_3$ be the 3-vertex path, with vertices labeled as 1, 2, and 3 from left to right. Suppose we start with the following pyramid, which has two pieces coming from vertex 1, and the piece coming from vertex 2 as the unique source:
    \begin{center}
        \includegraphics[width=3.5cm]{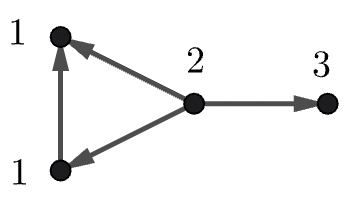}
    \end{center}
    The possible words associated to this pyramid are 2113, 2131, and 2311, since the 2 must come before the other pieces, but the 3 commutes with the 1's since they are nonadjacent vertices. The standard word is $\boxed{2311}$, since it is lexicographically maximal among the words for this pyramid. Given any word $w$, write $[w]$ for the heap corresponding to the equivalence class of $w,$ so our current heap is $[2311].$

    The equivalence class of rotations of $[2311]$ will consist of 4 pyramids, one with each of the 4 pieces as the source. To find these 4 rotations of $[2311]$, we will first rotate so that 3 becomes the source. To do so, we move 3 and all vertices reachable from it (of which there are none) to the end of the word, giving the word 2113. Then we take the 3 and everything following it (in this case, nothing), and move that to the start of the word,  giving the new word $A_3([2311])=A_3([2113]) = [3211]$. The corresponding heap is shown below:
    \begin{center}
        \includegraphics[width=3.5cm]{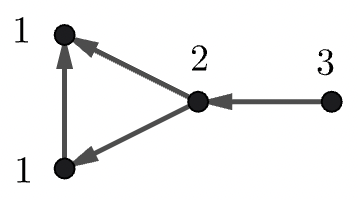}
    \end{center}
    This heap is already a pyramid with 3 as the source, which means that applying $A_3$ to it again will keep it the same, so in fact $A_3([2311])=A_3^2([2311])=\dots=A_3^\infty([2311])=[3211].$ Thus, $[3211]$ is the pyramid in the equivalence class that has 3 as the source. Note also that $\boxed{3211}$ is the only possible word associated to this pyramid, hence it is the standard word.

    Next, we need to move each of the 1's to the base to find the other two pyramids in the equivalence class. To help distinguish the two 1's, we will sometimes write $1_b$ for the bottom 1 and $1_t$ for the top 1 (although the actual words will just use the same letter 1 twice, and the heap would be considered the same if we swapped $1_b$ with $1_t$). First we will rotate so that $1_b$ becomes the source. To apply $A_{1_b}$ to $[3211]=[321_b1_t],$ we move the $1_b1_t$ part to the front, giving $A_{1_b}([321_b1_t])=[1_b1_t32]=[1132].$ This gives the following heap:
    \begin{center}
        \includegraphics[width=3.5cm]{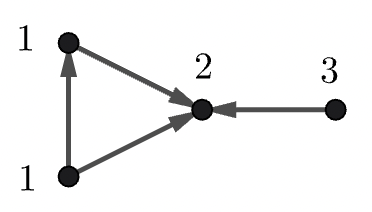}
    \end{center}
    This heap is not a pyramid, because $1_b$ and 3 are both sources. Its associated words are 1132, 1312, and 3112, so its standard word is 3112. Since it is not yet a pyramid, we apply $A_{1_b}$ a second time to try to it make it so that $1_b$ is the unique source. First, we again move $1_b$ and all pieces reachable from it (in this case, $1_b1_t2$), to the end of the word, and then we move that portion to the start. This gives $A_{1_b}([31_b1_t2])=[1_b1_t23]=[1123].$ We get the heap shown below:
    \begin{center}
        \includegraphics[width=3.5cm]{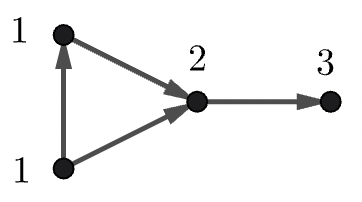}
    \end{center}
    This is now a pyramid with $1_b$ as the unique source, so it is the pyramid in the equivalence class that has $1_b$ as the source. There, is only one possible word, 1123, so $\boxed{1123}$ is the standard word for this pyramid.

    Finally, we rotate so that $1_t$ becomes the source. To apply $A_{1_t}$ to $[1_b1_t23]$, we move the $1_t23$ portion of the word $1_b1_t23$ to the start, giving the word $1_t231_b$, so $A_{1_t}([1_b1_t23])=[1_t231_b]$. The heap is shown below:
    \begin{center}
        \includegraphics[width=3.5cm]{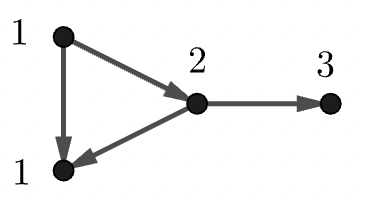}
    \end{center}
    This heap is a pyramid with $1_t$ as the unique source, so we do not need to apply $A_{1_t}$ again. The possible words for this pyramid are 1213 and 1231, so $\boxed{1231}$ is the standard word.

    Finally, to determine which of the 4 heaps in the equivalence class is the Lyndon heap, we compare the 4 standard words 2311, 3211, 1123, and 1231. Since 1123 is lexicographically minimal among these standard words, $\boxed{[1123]}$ is the Lyndon heap. 
    
    This example shows that having a piece from the smallest vertex as the source is a necessary but not sufficient condition for a heap to be a Lyndon heap, since both $[1123]$ and $[1231]$ have one of the 1's as the source, but only $[1123]$ os a Lyndon heap. Note also that the 4 rotations are not just cyclic shifts of the associated standard words, since 3211 is not a cyclic shift of the other standard words.
\end{example}

Next, we give an example of how some the $\ol{p}$-coefficients for a specific graph can be interpreted in terms of Lyndon heaps:

\begin{example}
    In \cite{pierson2024power}, we listed the first few terms in the $\ol{p}$-expansion for the Kromatic symmetric function $\ol{X}_{K_2}$ of the 2-vertex clique $K_2$, consisting of 2 vertices with an edge between them: 
    \begin{align*}
        \ol{X}_{K_2} &= (- \ol{p}_2 + \ol{p}_{11}) + (2\ol{p}_3 - 2\ol{p}_{21}) + (-4\ol{p}_4 + 4 \ol{p}_{31} + \ol{p}_{22} - \ol{p}_{211}) \\
        &\ \ \ + (6\ol{p}_5 - 8 \ol{p}_{41} - 2\ol{p}_{32} + 2 \ol{p}_{311} + 2\ol{p}_{221}) + \cdots.
    \end{align*}
    We can now interpret these terms in terms of Lyndon heaps. Label the two vertices as 1 and 2. First we list the possible Lyndon heaps of each size from 1 to 5:
    \begin{enumerate}
        \item $[1]$ and $[2]$ are the 2 possible size 1 Lyndon heaps. 
        \item $[12]$ is the only size 2 Lyndon heap. $[11]$ and $[22]$ are not Lyndon heaps since they are periodic, and $[21]$ is also not a Lyndon heap since it is a rotation of $[12]$ and $[12]$ has a lexicographically smaller standard word. 
        \item The 2 possible size 3 Lyndon heaps are $[112]$ and $[122],$ since $[111]$ and $[222]$ are periodic, and all other size 3 pyramids are rotations of either $[112]$ or $[122].$ 
        \item The 3 possible size 4 Lyndon heaps are $[1112],$ $[1122]$, and $[1222].$ Note that $[1212]$ is not allowed since it is periodic.
        \item The 6 possible size 5 Lyndon heaps are $[11112],$ $[11122],$ $[11222],$ $[11212],$ $[12122],$ and $[12222].$
    \end{enumerate}
    For $K_2$ (and in fact, for all cliques), Lyndon heaps are actually the same as Lyndon words, i.e. aperiodic words that are lexicographically minimal among their rotations. This is because for heaps on cliques, all vertices are adjacent and so none of them can move past each other in the words, meaning there is actually only one possible word associated to each heap.
    
    From Theorem \ref{thm:XG_old_ps}, we can associate each $\ol{p}$-term in $\ol{X}_{K_2}$ to a collection of Lyndon heap coverings of $K_2$ where the heap sizes match the parts, such that for even parts, the heap sizes may be divided by a power of 2 as long as they stay even, and for odd parts, no Lyndon heaps may be repeated (while repeated even Lyndon heaps are allowed). We can ignore the signs, since we know that the $\ol{p}_\lam$ term has sign $(-1)^{|\lam|-\ell(\lam)}$. The Lyndon heap coverings associated to each of the $\ol{p}$-terms above are shown in the table below:
    $$\begin{array}{|cc|cc|}
    \hline
        -\ol{p}_2 & \{[12]\} & \ol{p}_{11} & \{[1],[2]\} \\ \hline 2\ol{p}_3 & \{[112]\},\{[122]\} & -2\ol{p}_{21} & \{[12],[1]\},\{[12],[2]\} \\
        \hline
        -4\ol{p}_4 & \{[1112]\},\{[1122]\},\{[1222]\},\{[12]\} & 4\ol{p}_{31} & \{[112],[1]\},\{[112],[2]\},\{[122],[1]\},\{[122],[2]\} \\ \hline
        \ol{p}_{22} & \{[12],[12]\} & -\ol{p}_{211} & \{[12],[1],[2]\} \\ \hline
        6\ol{p}_5 & \{[11112]\},\{[11122]\},\{[11212]\}, & -8\ol{p}_{41} & \{[1112],[1]\},\{[1112],[2]\},\{[1122],[1]\},\{[1122],[2]\}, \\
        & \{[11222]\},\{[12122]\},\{[12222]\} && \{[1222],[1]\},\{[1222],[2]\},\{[12],[1]\},\{[12],[2]\} \\ \hline
        -2\ol{p}_{32} & \{[112],[12]\},\{[122],[12]\} & 2\ol{p}_{311} & \{[112],[1],[2]\},\{[122],[1],[2]\} \\ \hline
        2\ol{p}_{221} & \{[12],[12],[1]\},\{[12],[12],[2]\} && \cdots \\ \hline
    \end{array}$$
    Note that in cases where the partition has a part of size 4, we can use a Lyndon heap of size 2 instead of a Lyndon heap of size 4, since $4/2^1=2$ is even. Note also that we see some instances where the size 2 heap $[12]$ is repeated, while the size 1 heaps $[1]$ and $[2]$ never get repeated, since even-sized heaps are allowed to be repeated while odd-sized heaps are not.
\end{example}

We will now introduce several heap generating series, which will be used throughout our proofs. A \emph{\tb{\tcb{trivial heap}}} is a heap where all vertices in the corresponding clan graph are nonadjacent, or equivalently, all pieces commute. Note that a trivial heap must use each vertex of $G$ at most once, since multiple copies of the same vertex are automatically connected in the clan graph. Trivial heaps are equivalent to independent sets in $G$. We write $\mc{T}_G$ for the set of trivial heaps on $G$ and $\mc{T}_G(n)$ for set of trivial heaps of size $n$ on $G$.

If the variables $\mathbf{x} = (x_1,\dots,x_{|V|})$ correspond to the vertices of $G$ and $H$ is a heap of type $\alpha,$ write $\mathbf{x}^{H} \coloneqq x_1^{\alpha_1}\dots x_{|V|}^{\alpha_{|V|}}.$ Then one can define power series $$H_G(t) \coloneqq \sum_{H\in \mc{H}_G}t^{|H|}\mathbf{x}^{H},\hspace{1cm}I_G(t)\coloneqq \sum_{T\in\mc{T}_G}t^{|T|}\mathbf{x}^{T},\hspace{1cm}P_G(t) \coloneqq \sum_{P\in \mc{P}_G}\frac{t^{|P|}\mathbf{x}^{P}}{|P|}$$ as the generating series for all heaps, trivial heaps, and pyramids, respectively. We use the notation $I_G(t)$ for the generating series for trivial heaps because it is the same as the \emph{\tb{\tcb{independence polynomial}}} whose coefficients count independent sets of each size.

When restricted to $|V|$ variables, the classic power series $$H(t) \coloneqq \sum_{n\ge 0} t^n h_n,\hspace{1cm}E(t) \coloneqq \sum_{n\ge 0}t^n e_n,\hspace{1cm}P(t) \coloneqq \sum_{n\ge 1} \frac{t^np_n}{n}$$ from symmetric function theory for the homogeneous, elementary, and power sum symmetric functions are the special case of $H_G(t)$, $I_G(t)$, and $P_G(t)$ where $G$ has no edges. This is because if $G$ is edgeless, then heaps are arbitrary multisets of $n$ vertices (corresponding to $h_n$), trivial heaps are sets of $n$ vertices with no repeats (corresponding to $e_n$), and pyramids are sets of the same vertex $n$ times, since otherwise there would be multiple sources (corresponding to $p_n$). The classic identities satisfied by $H(t)$, $E(t)$, and $P(t)$ are also satisfied by these new power series, namely, 
\begin{equation}\label{eqn:heap_series}
    H_G(t) = \frac1{I_G(-t)},\hspace{1cm}P_G(t) = \ln (H_G(t)) = -\ln(-I_G(t)),
\end{equation}
and these identities will be very important for our proofs. We will also use the identities 
\begin{equation}\label{eqn:H_E_factors}
    H(t) = \prod_{i\ge 1}\frac1{1-tx_i} = \prod_{i\ge 1}(1+tx_i+t^2x_i^2+\dots),\hspace{1cm}E(t) = \prod_{i\ge 1} (1+tx_i),
\end{equation}
which imply the special case $H(t) = 1/E(-t)$ of $H_G(t) = 1/I_G(-t).$ Additionally, we will use the fact that 
\begin{equation}\label{eqn:H_E_omega}
    \omega(H(t)) = \sum_{n\ge 0} t^n\omega(h_n) = \sum_{n\ge 0}t^n e_n = E(t).
\end{equation}
In what follows, we will not need variables to track vertices and so we will set those variables to 1, giving $$H_G(t) = \sum_{n\ge 0} |\mc{H}_G(n)|\cdot t^n,\hspace{1cm} I_G(t)=\sum_{n\ge 0}|\mc{T}_G(n)|\cdot t^n,\hspace{1cm}P_G(t) = \sum_{n\ge 1}\frac{|\mc{P}_G(n)|\cdot t^n}{n}.$$ We introduced the vertex variables just to help draw the connection with the power series $H(t),$ $E(t)$, and $P(t)$ for symmetric functions.

\section{Proofs of the expansion formulas}\label{sec:dirichlet}

Our proofs will start from our equation in \cite{pierson2024power}:
\begin{equation}\label{eqn:XG}                      
    \ol{X}_G(x_1,x_2,\dots) = \sum_{W \se V} (-1)^{|V\bs W|} \prod_{i\ge 1}I_{G|_W} (x_i).
\end{equation}

\begin{proof}[Proof of (\ref{eqn:XG})]
    $I_{G}(x_i)$ encodes ways to choose an independent set of vertices to assign color $x_i$ to. Multiplying over all colors corresponds to assigning every color to an independent set of vertices. The problem, however, is that some vertices might not get any colors. Taking an alternating sum over subsets $W\se V$ fixes that problem by canceling all terms for colorings where not all vertices get colored. To see this, note that colorings involving all vertices do not get canceled because they only show up for $W=V.$ For a coloring where $x_i$ is the first color missing, we can use a sign-reversing involution to pair up and cancel all terms corresponding to that coloring, because each subset $W$ containing all used vertices but not $x_i$ can be paired with the subset $W \cup \{x_i\},$ which still contains all used vertices.
\end{proof}

Our other key fact is that applying $\omega$ to both sides gives \begin{equation}\label{eqn:omega_XG}
    \omega(\ol{X}_G)(x_1,x_2,\dots) = \sum_{W\se V} (-1)^{|V\bs W|} \prod_{i\ge 1} H_{G|_W}(x_i),
\end{equation} 
which follows from the fact that $H_{G|_W}(x) = 1/I_{G|_W}(-x)$ and the following lemma from \cite{bernardi2020combinatorial} (which we will also use in \S \ref{sec:direct} to compute the $\omega(\ol{p})$ and $\omega(\ol{p}')$ bases):

\begin{lemma}[\cite{bernardi2020combinatorial}]\label{lem:omega}
    For any polynomial $F(x)$ with constant term 1, $$\omega\left(\prod_{i\ge 1}F(x_i)\right) = \prod_{i\ge 1}\frac1{F(-x_i)}.$$
\end{lemma}

\begin{proof}
    Working over the algebraic closure of the underlying field, we can find a factorization $$F(x) = (1+t_1x)\dots(1+t_kx).$$ By (\ref{eqn:H_E_factors}) and (\ref{eqn:H_E_omega}), for each $t_j$ we have $$\omega\left(\prod_{i\ge 1}(1+t_jx_i)\right) = \omega(H(t_j)) = E(t_j) = \prod_{i\ge 1}\frac1{1-t_jx_i}.$$ Multiplying over $t_j,$ we get $$\omega\left(\prod_{i\ge 1}F(x_i)\right) = \prod_{j=1}^k \omega\left(\prod_{i\ge 1}(1+t_jx_i)\right) = \prod_{j=1}^k \prod_{i\ge 1}\frac1{1-t_jx_i} = \prod_{i\ge 1}\frac1{(1-t_1x_i)\dots(1-t_kx_i)} = \prod_{i\ge 1} \frac1{F(-x_i)},$$ as claimed.
\end{proof}

The idea of all our proofs will be to factor the $\prod_{i\ge 1}I_{G|_W}(x_i)$'s or $\prod_{i\ge 1}H_{G|_W}(x_i)$'s into a product of factors of the form $(1+\ol{p}_k)$ or $(1+\ol{p}'_k)$. We will get four formulas
\begin{align*}
    \ol{X}_G = \sum_{W\se V}(-1)^{|V\bs W|}\prod_{k \ge 1}(1+\ol{p}_k)^{a_W(k)}, \\
    \omega(\ol{X}_G) = \sum_{W\se V}(-1)^{|V\bs W|}\prod_{k \ge 1}(1+\ol{p}_k)^{b_W(k)}, \\
    \ol{X}_G = \sum_{W\se V}(-1)^{|V\bs W|}\prod_{k \ge 1}(1+\ol{p}'_k)^{c_W(k)}, \\
    \omega(\ol{X}_G) = \sum_{W\se V}(-1)^{|V\bs W|}\prod_{k \ge 1}(1+\ol{p}'_k)^{d_W(k)},
\end{align*}
and then we will interpret the resulting coefficients combinatorially. (The first of the above formulas is the same one we gave in \cite{pierson2024power}.) As in \S\ref{sec:direct}, we will prove our results in the opposite order from the statements in \S \ref{sec:intro} and in the above list, as the later ones are easier to prove. In \S \ref{sec:subgraphs} we will use one of these expansion formulas to prove Corollary \ref{cor:subgraphs}.

\subsection{Proof of Theorem \ref{thm:omega_XG_new_ps}: \texorpdfstring{$\ol{p}'$}{p'}-expansion of \texorpdfstring{$\omega(\ol{X}_G)$}{omega(XG)}}\label{sec:omega_XG_new_ps}

Since $\omega(\ol{X}_G)$ is a sum of products of $H_{G|_W}$'s by (\ref{eqn:omega_XG}), we will find a factorization for each $W$ of the form $$\prod_{k\ge 1}(1 + \ol{p}'_k)^{d_W(k)} = \prod_{i\ge 1} H_{G|_W}(x_i).$$ Since $1+\ol{p}'_k = \prod_{i\ge 1}1/(1 - x_i^k),$ we can factor both sides over the $x_i$ variables and just pull out a single $x_i,$ so the problem reduces to finding a factorization of the form
\begin{equation}\label{eqn:H_new_p_factors}
    \prod_{k\ge 1} \frac1{(1-x_i^k)^{d_W(k)}} = H_{G|_W}(x_i),
\end{equation} 
where the unknown part is the exponents $d_W(k).$

\begin{claim}\label{claim:d}
    $d_W(k) = |\mc{L}_{G|_W}(k)|$ counts Lyndon heaps of size $k$ on $G|_W.$
\end{claim}

\begin{proof}
We take logarithms on both sides of (\ref{eqn:H_new_p_factors}) and apply (\ref{eqn:heap_series}) to get $$-\sum_{k\ge 1}d_W(k)\cdot \ln(1 - x_i^k) = P_{G|_W}(x_i).$$ Expanding the left side using the power series for log and the right side using the definition gives $$\sum_{k\ge 1} \sum_{d\ge 1}d_W(k)\frac{(x_i^k)^d}{d} = \sum_{n\ge 0}\frac{|\mc{P}_{G|_W}(n)|}{n}x_i^n.$$ Taking coefficients of $x_i^n$ on both sides (so $d = n/k$) gives $$\sum_{k\mid n} \frac{d_W(k)}{n/k} = \frac{|\mc{P}_{G|_W}(n)|}{n},$$ and multiplying by $n$ gives $$\sum_{k\mid n} k\cdot d_W(k) = |\mc{P}_{G|_W}(n)|.$$ Now, it is known from \cite{lalonde1995lyndon} that each pyramid on $n$ vertices can be uniquely written as a rotation of a single Lyndon heap or as a power $H^{d}$ of rotation $H$ of some Lyndon heap of size $k = n/d.$ Every Lyndon heap of size $k$ has $k$ rotations, so there are $k\cdot|\mc{L}_{G|_W}(k)|$ different pyramids of size $n$ that come from Lyndon heaps of each size $k\mid n.$ It follows that 
\begin{equation}\label{eqn:L_to_P}
    \sum_{k\mid n} k\cdot|\mc{L}_{G|_W}(k)| = |\mc{P}_{G|_W}(n)|.
\end{equation} But then Claim \ref{claim:d} follows by induction on $n,$ because if we assume $d_W(k) = |\mc{L}_{G|_W}(k)|$ for all $k<n,$ then $$n\cdot d_W(n) = |\mc{P}_{G|_W}(n)| - \sum_{k\mid n, k<n} k\cdot |\mc{L}_{G|_W}(k)| = n\cdot|\mc{L}_{G|_W}(n)|.$$ Thus by induction, $d_W(k) = |\mc{L}_{G|_W}(k)|$ for all $k,$ as claimed.
\end{proof}

Now multiplying over the $x_i$ variables and taking the alternating sum over subsets $W$, we get $$\omega(\ol{X}_G) = \sum_{W \se V} (-1)^{|V\bs W|}\prod_{k\ge 1}(1+\ol{p}_k')^{|\mc{L}_{G|_W}(k)|}.$$ Next, we note that within each $\prod_{k\ge 1}(1+\ol{p}_k')^{|\mc{L}_{G|_W}(k)|}$ term, each $\ol{p}'_\lam$ has coefficient $\prod_{k\ge 1} \binom{|\mc{L}_{G|_W}(k)|}{i_k(\lam)}$, where $i_k(\lam)$ is the number of times $k$ occurs as a part in $\lam.$ This product is precisely the number of ways to choose a Lyndon heap on $G|_W$ of size $\lam_i$ for each $i=1,\dots,\ell(\lam),$ such that the Lyndon heaps chosen for repeated parts of $\lam$ are distinct and unordered. Then, as in (\ref{eqn:XG}), taking the alternating sum over subsets $W \se V$ cancels all Lyndon heaps that do not use all the vertices, so we are left with the coefficient of $\ol{p}'_\lam$ counting sets of distinct Lyndon heaps of sizes $\lam_1,\dots,\lam_{\ell(\lam)}$ that cover all vertices in $V$, as claimed in Theorem \ref{thm:omega_XG_new_ps}. \qed

\subsection{Proof of Theorem \ref{thm:XG_new_ps}: \texorpdfstring{$\ol{p}'$}{p'}-expansion of \texorpdfstring{$\ol{X}_G$}{XG}} \label{sec:XG_new_ps}

In this case we are interested in $\ol{X}_G$ and not $\omega(\ol{X}_G)$, and $\ol{X}_G$ is a sum of products of $I_{G|_W}$'s by (\ref{eqn:XG}). We will thus find a factorization of the form $$\prod_{k\ge 1}(1+\ol{p}'_k)^{c_W(k)} = \prod_{i\ge 1}I_{G|_W}(x_i).$$ Factoring over the $x_i$ variables as before, this reduces to finding the exponents $c_W(k)$ such that \begin{equation}\label{eqn:I_new_p_factors}
    \prod_{k\ge 1} \frac{1}{(1 - x_i^k)^{c_W(k)}}  = I_{G|_W}(x_i).
\end{equation}

\begin{claim}\label{claim:c}
    $c_W(k) = \begin{cases}
        (-1)^{k+1}|\mc{L}_{G|_W}(k)| &\tn{if }k\not\equiv 2\pmod4, \\
        -|\mc{L}_{G|_W}(k)|-|\mc{L}_{G|_W}(k/2)| &\tn{if }k\equiv 2\pmod4.
    \end{cases}$
\end{claim}

\begin{proof}
    Like in \S \ref{sec:omega_XG_new_ps}, we take logarithms on both sides of (\ref{eqn:I_new_p_factors}) and apply (\ref{eqn:heap_series}) to get $$-\sum_{k\ge 1} c_W(k)\cdot\ln(1 - x_i^k) = -P_{G|_W}(-x_i),$$ which is the same as in \S \ref{sec:omega_XG_new_ps} except with some negative signs on the right. Expanding the power series on both sides gives $$\sum_{k\ge 1} \sum_{d\ge 1}c_W(k) \frac{(x_i^k)^d}{d} = \sum_{n\ge 0}(-1)^{n+1}\frac{|\mc{P}_{G|_W}(n)|}{n}x_i^n.$$ Taking the coefficient of $x^n$ on both sides and then multiplying by $n$ gives $$\sum_{k\mid n}k\cdot c_W(k) = (-1)^{n+1}|\mc{P}_{G|_W}(n)|.$$ Applying M\"obius inversion gives \begin{equation}\label{eqn:c_mobius}
        k\cdot c_W(k) = \sum_{d\mid k}\mu\left(\frac{k}{d}\right)\cdot(-1)^{d+1}\cdot|\mc{P}_{G|_W}(d)|.
    \end{equation}
    We can also apply M\"obius inversion to equation (\ref{eqn:L_to_P}) from \S \ref{sec:omega_XG_new_ps} to get \begin{equation}\label{eqn:L_mobius}
        k\cdot |\mc{L}_{G|_W}(k)| = \sum_{d\mid k} \mu\left(\frac{k}{d}\right)\cdot|\mc{P}_{G|_W}(d)|.
    \end{equation} For $k$ odd, (\ref{eqn:c_mobius}) and (\ref{eqn:L_mobius}) are the same since every $d\mid k$ is odd and thus $(-1)^{d+1}=1$ for all $d\mid k,$ so $c_W(k) = |\mc{L}_{G|_W}(k)|.$ For $k$ even, we get 
    \begin{align*}
        k\cdot c_W(k) &= -\sum_{d\mid k\tn{ even}}\mu\left(\frac{k}{d}\right)\cdot|\mc{P}_{G|_W}(d)|+\sum_{d\mid k\tn{ odd}}\mu\left(\frac{k}{d}\right)\cdot|\mc{P}_{G|_W}(d)| \\
        &= -\sum_{d\mid k}\mu\left(\frac{k}{d}\right)\cdot|\mc{P}_{G|_W}(d)| + 2\sum_{d\mid k\tn{ odd}}\mu\left(\frac{k}{d}\right)\cdot|\mc{P}_{G|_W}(d)| \\
        &= -k\cdot|\mc{L}_{G|_W}(k)|+2\sum_{d\mid k\tn{ odd}}\mu\left(\frac{k}{d}\right)\cdot|\mc{P}_{G|_W}(d)|.
    \end{align*}
    If $4\mid k,$ the second term goes away because whenever $d$ is odd, $k/d$ is a multiple of 4 and so $\mu(k/d) = 0.$ Thus, $c_W(k) = -|\mc{L}_{G|_W}(k)|$ when $4\mid k.$ Finally, if $k\equiv 2\pmod 4,$ the odd divisors of $k$ are precisely the divisors of $k/2,$ and for each such odd divisor $d,$ $\mu(k/d) = -\mu(k/(2d)),$ since $k/d$ has an extra factor of 2 compared to $k/(2d)$. Plugging that in and then applying (\ref{eqn:L_mobius}) with $k/2$ in place of $k$ gives $$2\sum_{d\mid k\tn{ odd}}\mu\left(\frac{k}{d}\right)\cdot|\mc{P}_{G|_W}(d)| = -2\sum_{d\mid (k/2)}\mu\left(\frac{k/2}d\right)\cdot|\mc{P}_{G|_W}(d)| = -2\cdot \frac k2 \cdot |\mc{L}_{G|_W}(k/2)| = -k\cdot|\mc{L}_{G|_W}(k/2)|.$$ Plugging that back in, we get that for $k\equiv 2\pmod4,$ $$k\cdot c_W(k) = -k\cdot|\mc{L}_{G|_W}(k)| - k\cdot|\mc{L}_{G|_W}(k/2)|,$$ and dividing both sides by $k$ finishes the proof of Claim \ref{claim:c}.
\end{proof} 

Now to finish proving Theorem \ref{thm:XG_new_ps}, we plug our formula for $c_W(k)$ into (\ref{eqn:I_new_p_factors}), multiply over colors, and take the alternating sum over subsets $W\se V$ to get $$\ol{X}_G = \sum_{W\se V}(-1)^{|V\bs W|}\prod_{k\not\equiv 2 \ (\tn{mod }4)}(1+\ol{p}'_k)^{(-1)^{k+1}|\mc{L}_{G|_W}(k)|}\prod_{k\equiv 2\ (\tn{mod }4)}(1+\ol{p}'_k)^{-|\mc{L}_{G|_W}(k)|-|\mc{L}_{G|_W}(k/2)|}.$$ The coefficient $[\ol{p}'_\lam]\ol{X}_G$ is thus $$\sum_{W\se V}(-1)^{|V\bs W|}\prod_{k\tn{ odd}}\binom{|\mc{L}_{G|_W}(k)|}{i_k(\lam)}\prod_{4\mid k}\binom{-|\mc{L}_{G|_W}(k)|}{i_k(\lam)}\prod_{k\equiv 2 \ (\tn{mod }4)}\binom{-|\mc{L}_{G|_W}(k)|-|\mc{L}_{G|_W}(k/2)|}{i_k(\lam)},$$ where again, $i_k(\lam)$ is the multiplicity of $k$ as a part of $\lam.$ For a negative binomial coefficient, we have $$\binom{-n}{r} = \frac{-n(-n-1)\dots(-n-r+1)}{r!} = (-1)^r\binom{n+r-1}{r} = (-1)^r\multiset{n}{r},$$ where $\multiset{n}{r} = \binom{n+r-1}{r}$ is the number of ways to choose $r$ things out of $n$ with repetition allowed. The $(-1)^r$ means we flip the sign for every even part of $\lam,$ which is equivalent to multiplying by $(-1)^{\lam_i-1}$ for each part of $\lam,$ since $(-1)^{\lam_i-1}$ is 1 for odd parts and $-1$ for even parts. Thus, in total, the sign of the $\ol{p}'_\lam$ coefficient for a fixed $W$ (ignoring the $(-1)^{|V\bs W|}$) is always $(-1)^{\lam_1-1+\dots+\lam_{\ell(\lam)}-1} = (-1)^{|\lam|-\ell(\lam)}$, so our overall coefficient $[\ol{p}_\lam']\ol{X}_G$ becomes $$(-1)^{|\lam|-\ell(\lam)}\sum_{W\se V}(-1)^{|V\bs W|} \prod_{k\tn{ odd}}\binom{|\mc{L}_{G|_W}(k)|}{i_k(\lam)}\prod_{4\mid k}\multiset{|\mc{L}_{G|_W}(k)|}{i_k(\lam)}\prod_{k\equiv 2 \ (\tn{mod }4)}\multiset{|\mc{L}_{G|_W}(k)|+|\mc{L}_{G|_W}(k/2)|}{i_k(\lam)}.$$ The part after the $(-1)^{|V\bs W|}$ counts ways to choose a Lyndon heap of size $\lam_i$ for each part $\lam_i$ (except that if $\lam_i\equiv2\pmod4$ the Lyndon heap may also have size $\lam_i/2$), such that the Lyndon heaps for repeated odd parts must be distinct while the ones for repeated even heaps may be the same. The alternating sum over $W\se V$ eliminates multisets of Lyndon heaps where not all vertices are used, so the coefficient $(-1)^{|\lam|-\ell(\lam)}[p'_\lam]\ol{X}_G$ is as claimed in Theorem \ref{thm:XG_new_ps}. \qed

\subsection{Proof of Theorem \ref{thm:omega_XG_old_ps}: \texorpdfstring{$\ol{p}$}{p}-expansion of \texorpdfstring{$\omega(\ol{X}_G)$}{omega(XG)}}\label{sec:omega_XG_old_ps}

Like in \S\ref{sec:omega_XG_new_ps}, we will find a factorization of the form $$\prod_{k\ge 1}(1+\ol{p}_k)^{b_W(k)} = \prod_{i\ge 1} H_{G|_W}(x_i).$$ Since $1+\ol{p}_k = \prod_{i\ge 1}(1+x_i^k),$ when factoring over $x_i$'s this becomes $$\prod_{k\ge 1}(1+x_i^k)^{b_W(k)} = H_{G|_W}(x_i).$$

\begin{claim}\label{claim:b}
    $\displaystyle{b_W(k) = \sum_{2^j\mid k}|\mc{L}_{G|_W}(k/2^j)|}.$
\end{claim}

\begin{proof}
    As before, we start by taking logarithms on both sides, giving $$\sum_{k\ge 1} b_W(k)\cdot \ln(1+x_i^k) = I_{G|_W}(x_i).$$ Expanding both sides as power series yields $$\sum_{k\ge 1}\sum_{d\ge 1} b_W(k)\cdot\frac{(-1)^{d+1}(x_i^k)^d}{d} = \sum_{n\ge 1} \frac{|\mc{P}_{G|_W}(n)|}{n}x^n.$$ Taking coefficients of $x^n$ on both sides and multiplying by $n$ gives 
    \begin{equation}\label{eqn:b_P}
        \sum_{k\mid n} (-1)^{n/k+1}\cdot k\cdot b_W(k) = |\mc{P}_{G|_W}(n)|.
    \end{equation}
    The left side is the Dirichlet convolution of the sequences $(-1)^{d+1}$ and $k\cdot b_W(k),$ so to apply M\"obius inversion, we need to find the Dirichlet inverse of $(-1)^{d+1}$:

    \begin{lemma}\label{lem:dirichlet}
        The Dirichlet inverse of $(-1)^{d+1}$ is $\hat{\mu}(n) = \begin{cases}
            \mu(n) & \tn{if }n\tn{ is odd}, \\
            2^{j-1}\mu(n/2^j) &\tn{if }n\tn{ is even with }n/2^j\tn{ odd}.
        \end{cases}$
    \end{lemma}

    \begin{proof}
        We need to show that $(1,-1,1,\dots)*(\hat{\mu}(1),\hat{\mu}(2),\hat{\mu}(3),\dots) = (1,0,0,\dots),$ where the $*$ is Dirichlet convolution. That is, we need to show that \begin{equation}\label{eqn:dirichlet}
            \sum_{d\mid n}(-1)^{n/d+1}\cdot \hat{\mu}(d) = \begin{cases}
            1 & \tn{if }n=1, \\
            0 &\tn{else}.
        \end{cases}
        \end{equation}
        This holds for $n=1.$ For $n>1$ odd, $(-1)^{n/d+1}=1$ all $d\mid n,$ so the left side of (\ref{eqn:dirichlet}) becomes $$\sum_{d\mid n} \mu(d) = 0,$$ as needed. For $n$ even with $n/2^j$ odd, each $d\mid n$ can be uniquely written as $d = 2^id'$ for some $i\le j$ and $d'\mid (n/2^j).$ Then $(-1)^{n/d+1}=1$ if and only if $i=j$ since that is the only case where $n/d$ is odd. Thus, if we gather the terms on the left side of (\ref{eqn:dirichlet}) based on the value of $d',$ we get
        \begin{align*}
            \sum_{d\mid n} (-1)^{n/d+1}\cdot \hat{\mu}(d) &= \sum_{d=2^jd'}\hat{\mu}(d) - \sum_{i=1}^{j-1} \ \sum_{d = 2^id'}\hat{\mu}(d) - \sum_{d = d'}\hat{\mu}(d) \\
            &= \sum_{d'\mid (n/2^j)}\left(2^{j-1}\mu(d')-\sum_{i=1}^{j-1}2^{i-1}\mu(d')-\mu(d')\right) \\
            &=\sum_{d'\mid (n/2^j)}\mu(d')(2^{j-1}-2^{j-2}-2^{j-3}-\dots-2-1-1) = 0.
        \end{align*}
        This proves Lemma \ref{lem:dirichlet}.
    \end{proof}

    Returning to Claim \ref{claim:b}, applying M\"obius inversion to (\ref{eqn:b_P}) gives 
    \begin{equation}\label{eqn:b_mobius}
        k\cdot b_W(k) = \sum_{k/d\tn{ odd}}\mu\left(\frac kd\right)\cdot|\mc{P}_{G|_W}(d)|+\sum_{k/(2^jd)\tn{ odd}, \ j\ge 1}2^{j-1}\mu\left(\frac k{2^jd}\right)\cdot |\mc{P}_{G|_W}(d)|. 
    \end{equation}
    From (\ref{eqn:L_mobius}), we can write 
    \begin{align*}
        k\cdot |\mc{L}_{G|_W}(k)| &= \sum_{d\mid k}\mu\left(\frac{k}{d}\right)\cdot|\mc{P}_{G|_W}(d)| \\
        &= \sum_{k/d\tn{ odd}} \mu\left(\frac kd\right)\cdot |\mc{P}_{G|_W}(d)| - \sum_{k/(2d)\tn{ odd}} \mu\left(\frac k{2d}\right)\cdot|\mc{P}_{G|_W}(d)|.
    \end{align*}
    The terms where $k/(2d)$ is even go away because if $k/(2d)$ is even, then $\mu(k/d)=0$ since $4\mid k/d.$ On the other hand, if $k/(2d)$ is odd, then $\mu(k/d) = -\mu(k/(2d))$ since $k/d$ has an extra factor of 2 compared to $k/(2d).$ Plugging that into (\ref{eqn:b_mobius}) gives 
    \begin{align*}
        k\cdot b_W(k) &= k\cdot|\mc{L}_{G|_W}(k)| + 2\sum_{k/(2d)\tn{ odd}}\mu\left(\frac{k}{2d}\right)\cdot|\mc{P}_{G|_W}(d)| + \sum_{k/(2^jd)\tn{ odd}, \ j\ge 2} 2^{j-1}\mu\left(\frac{k}{2^jd}\right)\cdot |\mc{P}_{G|_W}(d)|\\
        &= k\cdot|\mc{L}_{G|_W}(k)|\\
        &+2\left(\sum_{(k/2)/d\tn{ odd}}\mu\left(\frac{k/2}{d}\right)\cdot|\mc{P}_{G|_W}(d)|+\sum_{(k/2)/(2^{j'}d)\tn{ odd}, \ j'\ge1}2^{j'-1}\mu\left(\frac{k/2}{2^{j'}d}\right)\cdot|\mc{P}_{G|_W}(d)|\right),
    \end{align*}
    where we made the substitution $j'=j-1.$ Now we can use induction on $k.$ Claim \ref{claim:b} holds for $k$ odd because the last line above is 0. For $k$ even, if we assume Claim \ref{claim:b} holds for $k/2,$ then the last line above is precisely $2\cdot k/2\cdot b_W(k/2)$ by the inductive hypothesis, so 
    \begin{align*}
        k\cdot b_W(k) &= k\cdot|\mc{L}_{G|_W}(k)| + 2\cdot\frac k2\cdot b_W\left(\frac k2\right) \\
        &= k\cdot|\mc{L}_{G|_W}(k)| + k\sum_{2^j\mid (k/2)}\left|\mc{L}_{G|_W}\left(\frac{k/2}{2^j}\right)\right| \\
        &= k\cdot\sum_{2^j\mid k}\left|\mc{L}_{G|_W}\left(\frac{k}{2^j}\right)\right|,
    \end{align*}
    as claimed.
\end{proof}

Now to finish the proof of Theorem \ref{thm:omega_XG_old_ps}, we have $$\omega(\ol{X}_G) = \sum_{W\se V}(-1)^{|V\bs W|}\prod_{k\ge 1}(1+\ol{p}_k)^{\sum_{2^j\mid k}|\mc{L}_{G|_W}(k/2^j)|}.$$ For a particular $W$, the coefficient of $\ol{p}_\lam$ in the part after the $(-1)^{|V\bs W|}$ is thus $$\prod_{k\ge 1}\binom{\sum_{2^j\mid k}|\mc{L}_{G|_W}(k/2^j)|}{i_k(\lam)}.$$ This counts ways to choose for each $k\ge 1$ a total of $i_k(\lam)$ distinct Lyndon heaps with size $k/2^j$ for some $j\ge 0.$ Then taking the alternating sum over $W\se V$ eliminates sets of heaps where not all vertices are used, proving Theorem \ref{thm:omega_XG_old_ps}. \qed

\subsection{Proof of Theorem \ref{thm:XG_old_ps}: \texorpdfstring{$\ol{p}$}{p}-expansion of \texorpdfstring{$\ol{X}_G$}{XG}}\label{sec:XG_old_ps}

Starting as in our previous proofs, we will find a factorization of the form $$\prod_{k\ge 1}(1+\ol{p}_k)^{a_W(k)} = \prod_{i\ge 1}I_{G|_W}(x_i),$$ which reduces to finding a factorization of the form $$\prod_{k\ge 1}(1+x_i^k)^{a_W(k)} = I_{G|_W}(x_i).$$

\begin{claim}\label{claim:a}
    $a_W(k) = \begin{cases}
        |\mc{L}_{G|_W}(k)| & \tn{if }k\tn{ is odd}, \\
        -\displaystyle{\sum_{k/2^j\tn{ even}}}|\mc{L}_{G|_W}(k/2^j) & \tn{if }k\tn{ is even}.
    \end{cases}$
\end{claim}

\begin{proof}
    As in our previous proofs, we take logarithms to get $$\sum_{k\ge 1}a_W(k)\cdot\ln(1+x_i^k) = -P_{G|_W}(-x_i).$$ Expanding both sides gives $$\sum_{k\ge 1}\sum_{d\ge 1}a_W(k)\cdot\frac{(-1)^{d+1}(x_i^k)^d}{d} = \sum_{n\ge 1}(-1)^{n+1}\frac{|\mc{P}_{G|_W}(n)|}{n}x_i^n.$$ Taking coefficients of $x^n$ and multiplying by $n$ gives $$\sum_{k\mid n} (-1)^{n/k+1}\cdot k\cdot a_W(k) = (-1)^{n+1}|\mc{P}_{G|_W}(n)|.$$ Now we apply M\"obius inversion based on the Dirichlet inverse of $(-1)^{d+1}$ computed in Lemma \ref{lem:dirichlet} to get $$k\cdot a_W(k) = \sum_{k/d\tn{ odd}} \mu\left(\frac kd\right) (-1)^{d+1}|\mc{P}_{G|_W}(d)|+\sum_{k/(2^jd)\tn{ odd}, \ j\ge 1} 2^{j-1}\mu\left(\frac k{2^jd}\right)(-1)^{d+1}|\mc{P}_{G|_W}(d)|.$$ For $k$ odd, we are done because the second term goes away and $(-1)^{d+1}=1$ for all $d\mid k,$ so the first term equals $k\cdot |\mc{L}_{G|_W}(k)|$ by (\ref{eqn:L_mobius}).
    
    For $k$ even, $(-1)^{d+1}=-1$ unless $d$ is odd. If $2^{j'}$ is the largest power of 2 dividing $k$, then the terms above are exactly the negatives of the terms from (\ref{eqn:b_P}), except that the terms with $k/(2^{j'} d)$ odd have the wrong signs. Thus, we can subtract those terms and then add them back twice to exactly get $-k\cdot b_W(k)$. Then we can plug in the formula for $b_W(k)$ from \ref{claim:b}. Doing that gives 
    \begin{align*}
        k\cdot a_W(k) &= -k\cdot b_W(k) + 2\sum_{k/(2^jd)\tn{ odd}}2^{j-1}\mu\left(\frac k{2^jd}\right)(-1)^{d+1}|\mc{P}_{G|_W}(d)|\\
        &=-k\sum_{2^j\mid k}\left|\mc{L}_{G|_W}\left(\frac{k}{2^j}\right)\right| + 2^{j'}\sum_{d\mid (k/2^{j'})}\mu\left(\frac {k/2^{j'}}{d}\right)|\mc{P}_{G|_W}(d)|.
    \end{align*}
    But then we can apply (\ref{eqn:L_mobius}) to the last term with $k/2^{j'}$ in place of $k$ to get $$k\cdot a_W(k) = -k\sum_{2^j\mid k}\left|\mc{L}_{G|_W}\left(\frac{k}{2^j}\right)\right| + 2^{j'}\cdot \frac k{2^{j'}}\left|\mc{L}_{G|_W}\left(\frac k{2^{j'}}\right)\right| = -k\sum_{2^j\mid k, \ j\ne j'}\left|\mc{L}_{G|_W}\left(\frac{k}{2^j}\right)\right|,$$ and Claim \ref{claim:a} follows.
\end{proof}

To finish the proof of Theorem \ref{thm:XG_old_ps}, Claim \ref{claim:d} gives $$\ol{X}_G = \sum_{W\se V}(-1)^{|V\bs W|}\prod_{k\tn{ odd}}(1+ \ol{p}_k)^{|\mc{L}_{G|_W}(k)|}\prod_{k\tn{ even}}(1+\ol{p}_k)^{-\sum_{k/2^j\tn{ even}}|\mc{L}_{G|_W}(k/2^j)|}.$$ Thus, the coefficient of $\ol{p}_\lam$ in the part after the $(-1)^{|V\bs W|}$ is $$\prod_{k\tn{ odd}}\binom{|\mc{L}_{G|_W}(k)|}{i_k(\lam)}\prod_{k\tn{ even}}\binom{-\sum_{k/2^j\tn{ even}}|\mc{L}_{G|_W}(k/2^j)|}{i_k(\lam)}.$$ Replacing the negative binomial coefficient with a multichoose as in \S \ref{sec:XG_new_ps} pulls out a negative for each even part of $\lam,$ making the overall sign $(-1)^{|\lam|-\ell(\lam)}$ by the same reasoning as in \S \ref{sec:XG_new_ps}, so this becomes $$(-1)^{|\lam|-\ell(\lam)}\prod_{k\tn{ odd}}\binom{|\mc{L}_{G|_W}(k)|}{i_k(\lam)}\prod_{k\tn{ even}}\multiset{\sum_{k/2^j\tn{ even}}|\mc{L}_{G|_W}(k/2^j)|}{i_k(\lam)}.$$ The products after the $(-1)^{|\lam|-\ell(\lam)}$ count ways to choose a Lyndon heap of size $\lam_i$ for each odd part $\lam_i,$ and a Lyndon heap of some even size $\lam_i/2^j$ for each even part of $\lam,$ such that the Lyndon heaps for repeated odd parts are all distinct (since for odd parts we have a choose and not a multichoose). Then as always, the alternating sum over $W\se V$ pulls out the multisets of Lyndon heaps using all vertices of $G$, completing the proof of Theorem \ref{thm:XG_old_ps}. \qed

\subsection{Proof of Corollary \ref{cor:subgraphs}: \texorpdfstring{$\ol{X}_G$}{XG} and independence polynomials}\label{sec:subgraphs}

We want to prove that knowing $\ol{X}_G$ is equivalent to knowing the multiset of independence polynomials of induced subgraphs of $G$. First assume we are given the Kromatic symmetric function $\ol{X}_G$ of a graph. Then we can apply $\omega$ to it and compute the $\ol{p}$-expansion, which, as proven in \S\ref{sec:omega_XG_old_ps}, can be written in the form $$\omega(\ol{X}_G) = \sum_{W\se V}(-1)^{|V\bs W|} \prod_{k\ge 1}(1+\ol{p}_k)^{b_W(k)}$$ where the exponents $b_W(k)$ are nonnegative integers. For positive integers $K$ and $M$, write $\textsf{Sym}(K,M)$ for the finite dimensional subspace of $\Sym$ of dimension $M^K$ spanned by the basis vectors $\ol{p}_\lam = \ol{p}_1^{i_1(\lam)}\dots \ol{p}_K^{i_K(\lam)}$ where $0\le m_1(\lam),\dots,m_K(\lam) \le M,$ and let $\pi_{K,M}$ be the projection of $\Sym$ onto $\textsf{Sym}(K,M).$ If we choose $M \ge \max\{b_W(k):k\le K,W\se V\}$, then the projection of $\omega(\ol{X}_G)$ onto $\textsf{Sym}(K,M)$ is $$\pi_{K,M}(\omega(\ol{X}_G)) = \sum_{W\se V}(-1)^{|V\bs W|}\prod_{k=1}^K (1+ \ol{p}_k)^{b_W(k)}.$$ We claim now that the vectors $\ol{p}''_\lam \coloneqq (1+\ol{p}_1)^{i_1(\lam)}\dots(1+\ol{p}_K)^{i_K(\lam)}$ with $0\le m_1(\lam),\dots,m_K(\lam) \le M$ also form a basis for $\textsf{Sym}(K,M).$ First note that there are the right number of them, $M^K,$ and that they are all elements of $\textsf{Sym}(K,M).$ Next, fix any ordering on partitions such that $\lam < \mu$ whenever $|\lam| < |\mu|$. Then $\ol{p}''$ is a sum of the form $\ol{p}''_\lam = \ol{p}_\lam + \sum_{\mu < \lam} a_\mu \ol{p}_\mu,$ meaning the transition matrix from $\ol{p}_\lam$'s to $\ol{p}_\lam''$'s is triangular with 1's on the diagonal, so it is invertible, meaning the $\ol{p}_\lam''$'s form a basis. This means that we can compute the $\ol{p}_\lam''$ expansion of $\pi_{K,M}(\omega(\ol{X}_G)),$ which tells us all the exponents $b_W(k)$ for $1\le k \le K.$ By choosing $M$ and $K$ sufficiently large, we can similarly recover as many of the $b_W(k)$ exponents as desired, so the entire set of $b_W(k)$ exponents is uniquely determined by $\ol{X}_G$. But then knowing those exponents gives us the expansion $$\omega(\ol{X}_G) = \sum_{W\se V} (-1)^{|V\bs W|}\prod_{i\ge 1}H_{G|_W}(x_i),$$ since the $b_W(k)$'s were chosen to satisfy $\prod_{k\ge 1}(1+\ol{p}_k)^{b_W(k)} = \prod_{i\ge 1}H_{G|_W}(x_i).$ But now we can apply $\omega$ to each of these $H_{G|_W}(x_i)$ terms to get $$\ol{X}_G = \sum_{W\se V}(-1)^{|V\bs W|}\prod_{i\ge 1}I_{G|_W}(x_i),$$ which tells us the multiset of polynomials $I_{G|_W}(t)$ over all $W \se V,$ as claimed.

Conversely, given the multiset of independence polynomials $I_{G|_W}(t)$ over all induced subgraphs, we can simply multiply over the $x_i$ variables and then take the alternating sum over subsets $W\se V$ to recover $\ol{X}_G$. Thus, $\ol{X}_G$ contains exactly the same set of information about $G$ as the multiset of independence polynomials of induced subgraphs, as claimed in Corollary \ref{cor:subgraphs}. \qed

\section{Bijective proofs of the heap series factorizations}\label{sec:direct}

In this section we will give more direct bijective explanations for our heap series factorizations derived in \S \ref{sec:dirichlet},
\begin{align*}
    \prod_{i\ge 1} I_G(x_i) &= \prod_{k\tn{ odd}}(1+\ol{p}_k)^{|\mc{L}_G(k)|} \prod_{k\tn{ even}}(1+\ol{p}_k)^{-\sum_{k/2^j\tn{ even}}|\mc{L}_G(k/2^j)|}, \\
    \prod_{i\ge 1} H_G(x_i) &= \prod_{k \ge 1}(1+\ol{p}_k)^{\sum_{2^j\mid k}|\mc{L}_G(k/2^j)|}, \\
    \prod_{i\ge 1}I_G(x_i) &= \prod_{k\tn{ odd}}(1+\ol{p}'_k)^{|\mc{L}_G(k)|}\prod_{4\mid k}(1+\ol{p}_k')^{-|\mc{L}_G(k)|} \prod_{k\equiv 2\tn{ (mod 4)}}(1+\ol{p}'_k)^{-|\mc{L}_G(k)|-|\mc{L}_G(k/2)|}, \\
    \prod_{i\ge 1}H_G(x_i) &= \prod_{k\ge 1}(1+\ol{p}'_k)^{|\mc{L}_G(k)|},
\end{align*}
which we used to prove Theorems \ref{thm:XG_old_ps}, \ref{thm:omega_XG_old_ps}, \ref{thm:XG_new_ps}, and \ref{thm:omega_XG_new_ps}, respectively. As in \S \ref{sec:dirichlet}, we will give the proofs in reverse order so that the easier ones come first.

\subsection{Bijective proof of the factorization for Theorem \ref{thm:omega_XG_new_ps}: \texorpdfstring{$\ol{p}'$}{p'}-expansion of \texorpdfstring{$\omega(\ol{X}_G)$}{omega(XG)}}\label{sec:omega_XG_new_ps_2}

We consider first the simplest of our factorizations, 
\begin{equation}\label{eqn:HG_d}
    \prod_{i\ge 1}H_G(x_i) = \prod_{k\ge 1}(1+\ol{p}'_k)^{|\mc{L}_G(k)|},
\end{equation}
which we used to prove Theorem \ref{thm:omega_XG_new_ps}. The left side of (\ref{eqn:HG_d}) is the generating series for ways to assign a heap to each color. Since $$(1+\ol{p}_k')^{|\mc{L}_G(k)|} = \sum_{i\ge 0} \binom{|\mc{L}_G(k)|}{i}(\ol{p}'_k)^i,$$ the right side of (\ref{eqn:HG_d}) is the generating series for ways to choose a \emph{set} of distinct Lyndon heaps and then assign to each of them a nonempty \emph{multiset} of colors, since $\binom{|\mc{L}_G(k)|}{i}$ counts ways to choose $i$ distinct Lyndon heaps of a given size $k,$ and each $\ol{p}'_k$ factor is the generating series for all ways to give a nonempty set of colors to a given vertex set of size $k$ (in this case a size $k$ Lyndon heap).

To go from terms on the left side of (\ref{eqn:HG_d}) to terms on the right, recall from \S \ref{sec:intro} that each color heap can be uniquely decomposed into Lyndon factors (possibly repeated) $L_1\circ \dots \circ L_k$ with $L_1 \ge \dots \ge L_k.$ Then we can find the set of distinct Lyndon heaps that show up in at least one color set. Each of these Lyndon heaps has a multiset of colors assigned to it according to how many times it appears as a factor in each color heap, giving us the corresponding term on the right of (\ref{eqn:HG_d}). For the other direction, the original heap factorization can be recovered from this set of distinct Lyndon heaps and their associated multisets of colors, because we can reorganize the Lyndon heaps by color and then compose them in reverse lexicographic order within each color to find the color heaps. \qed

\subsection{Bijective proof of the factorization for Theorem \ref{thm:XG_new_ps}: \texorpdfstring{$\ol{p}'$}{p'}-expansion of \texorpdfstring{$\omega(\ol{X}_G)$}{omega(XG)}}\label{sec:XG_new_ps_2}

Next we consider the factorization 
\begin{equation}\label{eqn:IG_c}
    \prod_{i\ge 1}I_G(x_i) = \prod_{k\tn{ odd}}(1+\ol{p}'_k)^{|\mc{L}_G(k)|}\prod_{4\mid k}(1+\ol{p}_k')^{-|\mc{L}_G(k)|} \prod_{k\equiv 2\tn{ (mod 4)}}(1+\ol{p}'_k)^{-|\mc{L}_G(k)|-|\mc{L}_G(k/2)|} 
\end{equation}
that we used to prove Theorem \ref{thm:XG_new_ps}. We will first apply $\omega$ to both sides to make everything positive, which requires computing the $\omega(\ol{p}')$-basis. Since $1+\ol{p}'_k$ has leading coefficient 1, by Lemma \ref{lem:omega}, we can apply $\omega$ to it by plugging in $-x_i$ in place of $x_i$ for every $i$ and then taking the reciprocal, giving $$\omega(1+\ol{p}'_k) = \omega\left(\prod_{i\ge 1}\frac1{1-x_i^k}\right) = \prod_{i\ge 1}(1-(-x_i)^k) = \begin{cases}
    \prod_{i\ge 1}(1+x_i^k) & \tn{if }k\tn{ is odd}, \\
    \prod_{i\ge 1}\dfrac1{1+x_i^k + x_i^{2k} + \dots} & \tn{if }k\tn{ is even},
\end{cases}$$ which implies that $$\omega(1+\ol{p}_k') = \begin{cases}
    1+\ol{p}_k & \tn{if }k\tn{ is odd}, \\
    \dfrac1{1+\ol{p}'_k} & \tn{if }k\tn{ is even}.
\end{cases}$$ Thus, applying $\omega$ to both sides of (\ref{eqn:IG_c}) gives
\begin{equation}\label{eqn:HG_c}
    \prod_{i\ge 1}H_G(x_i) = \prod_{k\tn{ odd}}(1+\ol{p}_k)^{|\mc{L}_G(k)|}\prod_{4\mid k}(1+\ol{p}'_k)^{|\mc{L}_G(k)|}\prod_{k\equiv2\tn{ (mod 4)}}(1+\ol{p}'_k)^{|\mc{L}_G(k)|+|\mc{L}_G(k/2)|}.
\end{equation}

As explained in \S \ref{sec:omega_XG_new_ps_2}, the left side of (\ref{eqn:HG_c}) can be thought of as the generating series for ways to choose a set of distinct Lyndon heaps and assign a multiset of colors to each. On the right side of (\ref{eqn:HG_c}):
\begin{itemize}
    \item For each odd $k$ we are choosing a set of distinct Lyndon heaps of size $k$ and a \emph{set} of distinct colors for each Lyndon heap, since $\ol{p}_k$ enumerates ways to assign a nonempty \emph{set} of colors to a size $k$ heap.
    \item For each $4\mid k$ we are choosing a set of distinct Lyndon heaps of size $k$ and a \emph{multiset} of colors for each Lyndon heap, since $\ol{p}'_k$ enumerates ways to assign a nonempty \emph{multiset} of colors to a size $k$ heap.
    \item For each $k\equiv 2\pmod4$ we are choosing a set of distinct Lyndon heaps each of size either $k$ or $k/2$, and then assigning a \emph{multiset} of colors to each, such that each Lyndon heap of size $k/2$ gets each color in its multiset twice, since a $\ol{p}'_k$ term means we need to choose a multiset of colors and use each color $k$ times, which can be done by assigning it to each of the $k/2$ vertices twice. Equivalently, we can require that each Lyndon heap of size $k/2$ gets a multiset of colors where every color has even multiplicity. Note that since $k\equiv 2\pmod4,$ the heaps of size $k/2$ coming from this case are all odd in size.
\end{itemize}  
To go from a term on the right side of (\ref{eqn:HG_d}) to a term on the left, we need to group together repeated occurrences of the same Lyndon heap and take the union of their corresponding color sets. Each even sized Lyndon heap already shows up only once in a right hand side term, and it can already get any multiset of colors, so we really only need to consider the odd sized Lyndon heaps. Each odd Lyndon heap can show up at most once in the $\prod_{k\tn{ odd}}(1+\ol{p}_k)^{|\mc{L}_G(k)|}$ part, in which case it gets a \emph{set} of distinct colors, but it can also show up one additional time in the $\prod_{k\equiv2\tn{ (mod 4)}}(1+\ol{p}'_k)^{|\mc{L}_G(k)|+|\mc{L}_G(k/2)|}$ part, in which case it gets a \emph{multiset} of colors where every color has even multiplicity. If a Lyndon heap $L$ occurs in both places, we can thus say that the full color set assigned to it in the corresponding left hand side term will be the union of the two multisets. (When we take a union of multisets, for repeated elements we will find the multiplicity in the union by adding the multiplicities in the starting multisets.)

To go from a term on the left side of (\ref{eqn:HG_d}) to a term on the right side, we can again start by finding the corresponding set of distinct Lyndon heaps and their corresponding multisets of colors. That already tells us our set of distinct even sized Lyndon heaps and their associated multisets of colors. For each odd sized Lyndon heap $L$, we can break off from its multiset $S(L)$ of colors the \emph{set} of distinct colors that show up in the multiset an odd number of times, giving the corresponding factor from the $\prod_{k\tn{ odd}}(1+\ol{p}_k)^{|\mc{L}_G(k)|}$ part on the right. In the remainder of $S(L),$ each color shows up an even number of times, so we get the corresponding factor from the $\prod_{k\equiv2\tn{ (mod 4)}}(1+\ol{p}'_k)^{|\mc{L}_G(k)|+|\mc{L}_G(k/2)|}$ part. This uniquely tells us which term on the right side of (\ref{eqn:HG_d}) corresponds to a given term on the left side, so we have a bijection. \qed

\subsection{Bijective proof of the factorization for Theorem \ref{thm:omega_XG_old_ps}: \texorpdfstring{$\ol{p}$}{p}-expansion of \texorpdfstring{$\omega(\ol{X}_G)$}{omega(XG)}}\label{sec:omega_XG_old_ps_2}

Next we consider the factorization 
\begin{equation}\label{eqn:HG_b}
    \prod_{i\ge 1} H_G(x_i) = \prod_{k \ge 1}(1+\ol{p}_k)^{\sum_{2^j\mid k}|\mc{L}_G(k/2^j)|}
\end{equation} 
that we used to prove Theorem \ref{thm:omega_XG_old_ps}.

As discussed in \S \ref{sec:omega_XG_new_ps_2}, the left side of (\ref{eqn:HG_b}) can be thought of as the generating series for ways to choose a set of distinct Lyndon heaps then assign a nonempty multiset $S(L)$ of colors to each Lyndon heap $L$. The right side of (\ref{eqn:HG_b}) can be thought of as the generating series for ways to choose a set of distinct pairs $(L,j)$ with $L$ a Lyndon heap and $j$ a nonnegative integer, and then assign to each pair a multiset $S(L,j)$ of colors where every color has multiplicity exactly $2^j.$ Then, we can get an injection from terms on the right to terms on the left by merging together the color sets for all repeated copies of a Lyndon heap $L$, thus assigning to each distinct Lyndon heap $L$ the multiset union $S(L)\coloneqq\bigcup_{j\ge 0} S(L,j)$ of all its colors multisets, where again we add multiplicities when taking a union of multisets.

To show that this injection is in fact a bijection, we can describe the inverse map from terms on the left side of (\ref{eqn:HG_b}) to terms on the right. By \S \ref{sec:omega_XG_new_ps_2}, each term on the left corresponds to a set of distinct Lyndon heaps $L$ with a nonempty multiset $S(L)$ of colors assigned to each. For each distinct color $c$ occurring in $S(L),$ its multiplicity $m$ in $S(L)$ has a unique binary expansion, so $c$ needs to appear in the multisets $S(L,j)$ for precisely the values of $j$ such that $2^j$ shows up in the binary expansion of $m$. That lets us recover the unique term on the right side of (\ref{eqn:HG_b}) corresponding to each term on the left, as needed. \qed

\subsection{Bijective proof of the factorization for Theorem \ref{thm:XG_old_ps}: \texorpdfstring{$\ol{p}$}{p}-expansion of \texorpdfstring{$\ol{X}_G$}{XG}}\label{sec:XG_old_ps_2}

Finally, we give a bijective proof for the factorization 
\begin{equation}\label{eqn:IG_a}
    \prod_{i\ge 1} I_G(x_i) = \prod_{k\tn{ odd}}(1+\ol{p}_k)^{|\mc{L}_G(k)|} \prod_{k\tn{ even}}(1+\ol{p}_k)^{-\sum_{k/2^j\tn{ even}}|\mc{L}_G(k/2^j)|} 
\end{equation}
that we used to prove Theorem \ref{thm:XG_old_ps} (and hence to answer the original question from \cite{crew2023kromatic}). Like in \S \ref{sec:omega_XG_new_ps_2}, we will first make everything positive by applying $\omega$ to both sides, which requires determining how $\omega$ acts on the $\ol{p}$-basis. Since $1+\ol{p}_k$ has leading coefficient 1, by Lemma \ref{lem:omega} we can apply $\omega$ to it by plugging in $-x_i$ in place of each $x_i$ and then taking the reciprocal: $$\omega(1+\ol{p}_k) = \omega\left(\prod_{i\ge 1}(1+x_i^k)\right) = \prod_{i\ge 1}\frac1{1+(-x_i)^k} = \begin{cases}
    \prod_{i\ge 1}(1+x_i^k + x_i^{2k} + \dots) & \tn{if }k\tn{ is odd}, \\
    \prod_{i\ge 1}\dfrac1{1+x_i^k} & \tn{if }k\tn{ is even}.
\end{cases}$$
Equivalently, we can say that $$\omega(1+\ol{p}_k) = \begin{cases}
    1+\ol{p}'_k & \tn{if }k\tn{ is odd}, \\
    \dfrac1{1+\ol{p}_k} & \tn{if }k\tn{ is even}.
\end{cases}$$ Thus, applying $\omega$ to both sides, (\ref{eqn:IG_a}) becomes 
\begin{equation}\label{eqn:HG_a}
    \prod_{i\ge 1} H_G(x_i) = \prod_{k\tn{ odd}} (1+\ol{p}'_k)^{|\mc{L}_G(k)|} \prod_{k\tn{ even}}(1+\ol{p}_k)^{\sum_{k/2^j\tn{ even}}|\mc{L}_G(k/2^j)|}.
\end{equation}
We can now match up terms on the left side of (\ref{eqn:HG_a}) with terms on the right side as follows. 

We can again use \S \ref{sec:omega_XG_new_ps_2} to think of the left side as the generating series for ways to choose a set of distinct Lyndon heaps $L$ and a multiset $S(L)$ of colors for each. On the right side:
\begin{itemize}
    \item For each odd $k$, we need to choose a set of distinct Lyndon heaps of size $k$ and then assign a nonempty \emph{multiset} of colors to each of them.
    \item For each even $k,$ we need to choose a set of distinct Lyndon heaps each of some even size $k/2^j$, and then assign a \emph{set} of distinct colors to each Lyndon heap, such that a heap of size $k/2^j$ receives each color in its set $2^j$ times in order to make the exponent on each color used be $k$.
\end{itemize}
To go from terms on the right to terms on the left, note that in a term on the right side of (\ref{eqn:HG_a}), each odd sized Lyndon heap already only shows up once, but an even sized Lyndon heap $L$ may show up multiple times for different values of $j$, and for each such $j$ it will get a multiset of colors $S(L,j)$ where each color occurs exactly $2^j$ times. Thus, to recover the corresponding term on the left given a term on the right, we can assign to each distinct even heap showing up on the right side the union $S(L) \coloneqq \bigcup_{j\ge 0} S(L,j)$ of all multisets of colors assigned to it, which gives us a set of distinct Lyndon heaps and a multiset of colors for each.

Finally, to go from terms on the left side of (\ref{eqn:HG_a}) to terms on the right, we start with a set distinct Lyndon heaps and their associated multisets of colors. For each odd Lyndon heap, we are done since the right side also has each odd heap showing up only once but getting an arbitrary multiset of colors. For an even Lyndon heap $L$, we need to split its multiset $S(L)$ of colors into multisets $S(L,j)$ where each color in $S(L,j)$ has multiplicity $2^j.$ Like in \S \ref{sec:XG_old_ps_2}, this can always be done in exactly one way because for each color $c$, its multiplicity $m$ in $S(L)$ can uniquely be written in binary, and then $c$ needs to be a part of $S(L,j)$ for precisely the values of $j$ such that $2^j$ appears in the binary expansion of $m.$ Thus, we have a bijection between terms on the left and right sides of (\ref{eqn:HG_a}). \qed

\section{Quasisymmetric power sum expansions}\label{sec:quasi}

In this section, we will compute transition formulas from the classic $p$-basis to the $K$-analogues $\omega(\ol{p}'),$ $\ol{p}',$ $\omega(\ol{p}),$ and $\ol{p},$ and combine them with Theorem \ref{thm:athanasiadis} to compute the expansion of $\ol{X}_G(q)$ in each $K$-theoretic $p$-basis.

As noted in \cite{marberg2023kromatic}, we can write
\begin{equation}\label{eqn:XG_q}
    \ol{X}_G(q) = \sum_\alpha \frac1{[\alpha]_q!}X_{\textsf{Clan}_\alpha(G)}(q),
\end{equation}
where the sum is over all compositions $\alpha$ of length $|V|$ with all parts nonzero, and $$[\alpha]_q!\coloneqq \prod_{i=1}^{\ell(\alpha)}[\alpha_i]_q! = \prod_{i=1}^{\ell(\alpha)} \prod_{j=1}^{\alpha_i}(1+q+\dots+q^{j-1})$$ is the $q$-analogue of $\alpha! \coloneqq \alpha_1!\dots\alpha_{\ell(\alpha)}!,$ specializing to $\alpha!$ at $q=1.$ The reason for (\ref{eqn:XG_q}) is that a proper set coloring of $G$ is the same as a proper coloring of one of its $\alpha$-clan graphs, except that all repeat copies of the same vertex are considered identical, so if we assign $\alpha_i$ colors to them, it does not matter which one is assigned each color, and also based on Marberg's definition from \S \ref{sec:intro}, no ascents between duplicate copies of a vertex should count toward the total number of ascents. Hence, we need to divide by $[\alpha]_q!$ because each $[\alpha_i]_q!$ factor is the generating series in $q$ for the possible numbers of ascents within a clique of size $\alpha_i$ over the $\alpha_i!$ different ways that a given set of colors could be assigned to the vertices of the clique.

Combining (\ref{eqn:XG_q}) with Theorem \ref{thm:athanasiadis} gives 
\begin{equation}\label{eqn:XG_q_2}
    \ol{X}_G(q) = \sum_\lam \frac{p_\lam}{z_\lam}\sum_{|\alpha|=|\lam|}\frac1{[\alpha]_q!}\sum_{\pi \in \Pi_\lam(V^\alpha)} \ \sum_{o\in \mc{AO}^*(\textsf{Clan}_\alpha(G),\pi)}q^{\tn{asc}(o)},
\end{equation} 
where $V^\alpha$ is the vertex set of $\textsf{Clan}_\alpha(G).$ We can equivalently say that $o$ is a $G$-heap $H$ of type $\alpha$ and $\pi$ is a factorization $H = P_1\circ \dots \circ P_{\ell(\lam)}$ of $H$ into a list $P = (P_1,\dots,P_{\ell(\lam)})$ of pyramids with $|P_i|=\lam_i.$ Write $\mc{P}_G^{\tn{all }V}(\lam)$ for the set of all such pyramid lists $P$ that cover all vertices of $G$. Also, write $\tn{asc}(P)$ for the number of ascents of the corresponding heap $H$ such that ascents between duplicate copies of the same vertex do not count. Then we can re-frame (\ref{eqn:XG_q_2}) as
\begin{equation}\label{eqn:XG_q_final}
    \ol{X}_G(q) = \sum_\lam \frac{p_\lam}{z_\lam} \sum_{P\in \mc{P}^{\tn{all }V}_G(\lam)} q^{\tn{asc}(P)}.
\end{equation}

In the rest of this section, we will follow the same ordering as in \S \ref{sec:dirichlet} and \S \ref{sec:direct}, giving the $\omega(\ol{p}')$-expansion of $\ol{X}_G(q)$ in \S \ref{sec:omega_new_ps_q}, the $\ol{p}'$-expansion in \S \ref{sec:new_ps_q}, the $\omega(\ol{p})$-expansion in \S \ref{sec:omega_old_ps_q}, and finally the $\ol{p}$-expansion in \S \ref{sec:old_ps_q}.

\subsection{\texorpdfstring{$\ol{p}'$}{p'}-expansion for \texorpdfstring{$\omega(\ol{X}_G(q))$}{omega(XG(q))}}\label{sec:omega_new_ps_q}

We will first find the transition formula from the $p$ to $\ol{p}'$ basis. By (\ref{eqn:H_E_factors}), 
\begin{equation}\label{eqn:p_H}
    1+\ol{p}'_k = \prod_{i\ge 1}(1+x_i^k + x_i^{2k} + \dots) = \sum_{n\ge 1} h_n[p_k] = H(1)[p_k], 
\end{equation}
where $H(t)\coloneqq\sum_{t\ge 0}h_nt^n$ is the generating series for the homogeneous symmetric functions introduced in \S \ref{sec:background}, and the square brackets indicate \emph{\tb{\tcb{plethystic evaluation}}} $$f[p_k] \coloneqq f(x_1^k, x_2^k,\dots),$$ where we plug in the monomials $x_i^k$ of $p_k=x_1^k + x_2^k + \dots$ in place of the usual variables $x_i.$ Since we saw in \S \ref{sec:background} that $\ln(H(t)) = P(t)$, taking logarithms on both sides of (\ref{eqn:p_H}) gives $$\ln(1+\ol{p}_k') = P(1)[p_k] = \sum_{n\ge 1}\frac{p_n[p_k]}{n} = \sum_{n\ge 1}\frac{p_{nk}}{n}.$$ Now we use M\"obius inversion to isolate $p_k$, giving 
\begin{equation}\label{eqn:p_transition_1}
    p_k = \sum_{n\ge 1} \frac{p_{nk}}{n}\sum_{d\mid n}\mu(d) = \sum_{d\ge 1}\frac{\mu(d)}{d}\sum_{j\ge 1}\frac{p_{jdk}}{j} = \sum_{d\ge 1}\frac{\mu(d)}{d}\ln(1+\ol{p}_{dk}'),
\end{equation}
where we made the change of variables $j = n/d.$ We can expand the logarithm in (\ref{eqn:p_transition_1}) to get $$p_k = \sum_{d\ge 1}\frac{\mu(d)}{d}\sum_{n\ge 1}\frac{(-1)^{n+1}}{n}(\ol{p}'_{dk})^n.$$ Multiplying over the parts of a partition $\lam$ gives
\begin{equation}\label{eqn:p_coeff_1}
    p_\lam = \sum_{d,n}\prod_{i=1}^{\ell(\lam)}\frac{\mu(d_i)(-1)^{n_i+1}}{d_i n_i} \ol{p}'_{(d\lam)^n},
\end{equation} 
where $d=(d_1,\dots,d_{\ell(\lam)})$ and $n=(n_1,\dots,n_{\ell(\lam)})$ are now compositions of length $\ell(\lam)$ with all parts nonzero, and $(d\lam)^n$ is the partition where the $i$th part $\lam_i$ of $\lam$ is multiplied by $d_i$ and repeated $n_i$ times, and then the parts are rearranged in nonincreasing order.

Plugging (\ref{eqn:p_transition_1}) and (\ref{eqn:p_coeff_1}) into (\ref{eqn:XG_q_final}) gives

\begin{prop}\label{prop:XG_q_1}
If $G$ is the incomparability graph of a unit interval order, $$\omega(\ol{X}_G(q)) = \sum_\lam \frac1{z_\lam}\sum_{P \in \mc{P}^{\tn{all }V}_G(\lam)} q^{\tn{asc}(P)}\prod_{i=1}^{\ell(\lam)}\sum_{d\ge 1}\frac{\mu(d)}{d}\ln(1+\ol{p}'_{d\lam_i}),$$ and the $\ol{p}'$-coefficients for $\omega(\ol{X}_G(q))$ are $$\displaystyle{[\ol{p}'_\lam]\omega(\ol{X}_G(q)) = \sum_{(d\lam')^n=\lam}\frac1{z_{\lam'}}\sum_{P \in \mc{P}^{\tn{all }V}_G(\lam')} q^{\tn{asc}(P)}\prod_{i=1}^{\ell(\lam')}\frac{\mu(d_i)(-1)^{n_i+1}}{d_i n_i}}.$$
\end{prop}

Thus, unlike $\omega(\ol{X}_G)$ and $\omega(X_G(q))$, $\omega(\ol{X}_G(q))$ is in general not $\ol{p}'$-positive, since the polynomial in $q$ given in Proposition \ref{prop:XG_q_1} does not necessarily have all positive coefficients.

We can see how Proposition \ref{prop:XG_q_1} reduces to Theorem \ref{thm:omega_XG_new_ps} when $q=1$ as follows. When $q=1,$ Proposition \ref{prop:XG_q_1} becomes 
\begin{equation}\label{eqn:XG_from_q_1}
    \omega(\ol{X}_G) = \sum_\lam \frac{|\mc{P}^{\tn{all }V}_G(\lam)|}{z_\lam}\prod_{i=1}^{\ell(\lam)}\sum_{d\ge 1}\frac{\mu(d)}{d}\ln(1+\ol{p}'_{d\lam_i}).
\end{equation}
Write $\mc{P}_G(\lam)$ for the set of all lists $(P_1,\dots,P_{\ell(\lam)})$ of pyramids with $|P_i|=\lam_i,$ without the restriction that every vertex be used in at least one pyramid. Then we can rewrite (\ref{eqn:XG_from_q_1}) in terms of the $\mc{P}_G(\lam)$'s by using the same trick as before of taking an alternating sum over subsets of $V$:
$$\omega(\ol{X}_G) = \sum_{W\se V}(-1)^{|V\bs W|}\sum_\lam \frac{|\mc{P}_{G|_W}(\lam)|}{z_\lam}\prod_{i=1}^{\ell(\lam)}\sum_{d\ge 1}\frac{\mu(d)}{d}\ln(1+\ol{p}'_{d\lam_i}).$$ Now we can factor the part corresponding to a given $\lam$ and $W$, because $|\mc{P}_G(\lam)| = \prod_{i=1}^{\ell(\lam)}|\mc{P}_G(\lam_i)|,$ since the pyramids in the list can now be chosen independently of each other. Doing so gives $$\omega(\ol{X}_G) = \sum_{W\se V}(-1)^{|V\bs W|}\sum_\lam \frac1{z_\lam}\prod_{i=1}^{\ell(\lam)}|\mc{P}_{G|_W}(\lam_i)|\sum_{d\ge 1}\frac{\mu(d)}{d}\ln(1+\ol{p}'_{d\lam_i}).$$ Dividing by the factorials $i_k(\lam)!$ in $1/z_\lam$ is equivalent to saying that the pyramids corresponding to repeated parts of $\lam$ should be treated as unordered, meaning we can treat the whole list of pyramids as unordered. The remaining $1/\lam_i$ terms in $1/z_\lam$ can be grouped with their corresponding $|\mc{P}_{G|_W}(\lam_i)|$ terms. Letting $j=\lam_i$, the term of $\omega(\ol{X}_G)$ for a given $W$ is thus the sum of all products of unordered sets of terms of the form $$\frac1j|\mc{P}_{G|_W}(j)|\sum_{d\ge 1}\frac{\mu(d)}{d}\ln(1+\ol{p}_{dj}')$$ for $k\ge 1.$ But the generating series for all products of unordered sets of terms of a certain type is equivalent to the exponential of the sum of all the possible terms of that type, so $$\omega(\ol{X}_G) = \sum_{W \se V}(-1)^{|V\bs W|}\exp\left(\sum_{j\ge 1}\frac1j|\mc{P}_{G|_W}(j)|\sum_{d\ge 1}\frac{\mu(d)}{d}\ln(1+\ol{p}_{dj}')\right).$$ Reorganize the terms according to the value of $k=dj$ gives $$\omega(\ol{X}_G) = \sum_{W \se V}(-1)^{|V\bs W|}\exp\left(\sum_{k\ge 1}\ln(1+\ol{p}'_{k})\cdot \frac1k\sum_{d\mid k}\mu(d)\left|\mc{P}_{G|_W}\left(\frac kd\right)\right|\right).$$ But as we saw in equation (\ref{eqn:L_mobius}) from \S \ref{sec:XG_new_ps}, $$\frac1k\sum_{d\mid k} \mu(d)\left|\mc{P}_{G|_W}\left(\frac kd\right)\right| = |\mc{L}_{G|_W}(k)|,$$ hence our equation becomes 
\begin{align*}
    \omega(\ol{X}_G) &= \sum_{W\se V}(-1)^{|V\bs W|}\exp\left(\sum_{k\ge 1}\ln(1+\ol{p}_k')\cdot|\mc{L}_{G|_W}(k)|\right) \\
    &= \sum_{W\se V}(-1)^{|V\bs W|}\prod_{k\ge 1}(1+\ol{p}_k')^{|\mc{L}_{G|_W}(k)|},
\end{align*}
matching the factorization from \S \ref{sec:omega_XG_new_ps} and \S\ref{sec:omega_XG_new_ps_2} that we used to prove Theorem \ref{thm:omega_XG_new_ps}.

\subsection{\texorpdfstring{$\ol{p}'$}{p'}-expansion for \texorpdfstring{$\ol{X}_G(q)$}{XG(q)}}\label{sec:new_ps_q}

Since $\omega(p_\lam) = (-1)^{|\lam|-\ell(\lam)},$ we can use the same $p$-to-$\ol{p}'$ transition formula, except that each $\ol{p}'_\lam$ term will be multiplied by $(-1)^{|\lam'|-\ell(\lam')}$ for whichever $\lam'$ it comes from. So we can immediately deduce:

\begin{prop}\label{prop:new_p_XG_q}
If $G$ is the incomparability graph of a unit interval order, $$\ol{X}_G(q) = \sum_\lam \frac{(-1)^{|\lam|-\ell(\lam)}}{z_\lam}\sum_{P \in \mc{P}^{\tn{all }V}_G(\lam)} q^{\tn{asc}(P)}\prod_{i=1}^{\ell(\lam)}\sum_{d\ge 1}\frac{\mu(d)}{d}\ln(1+\ol{p}'_{d\lam_i}),$$ and the $\ol{p}'$-coefficients for $\ol{X}_G(q)$ are
$$\displaystyle{[\ol{p}'_\lam]\ol{X}_G(q) = \sum_{(d\lam')^n=\lam}\frac{(-1)^{|\lam'|-\ell(\lam')}}{z_{\lam'}}\sum_{P\in \mc{P}_G(\lam')} q^{\tn{asc}(P)}}\prod_{i=1}^{\ell(\lam')}\frac{\mu(d_i)(-1)^{n_i+1}}{d_i n_i}.$$
\end{prop}

We can relate Proposition \ref{prop:new_p_XG_q} to Theorem \ref{thm:XG_new_ps} by following the same approach as in \S \ref{sec:omega_new_ps_q}. We set $q=1$ and replace each $\mc{P}_G^{\tn{all }V}(\lam)$ term with an alternating sum of $\mc{P}_{G|_W}(\lam)$ terms to get $$\ol{X}_G = \sum_{W\se V}(-1)^{|V\bs W|} \sum_{\lam}\frac{(-1)^{|\lam|-\ell(\lam)}}{z_\lam}|\mc{P}_{G|_W}(\lam)|\prod_{i=1}^{\ell(\lam)}\sum_{d\ge 1}\frac{\mu(d)}{d}\ln(1+\ol{p}'_{d\lam_i}).$$ Now since $(-1)^{|\lam|-\ell(\lam)} = \prod_{i=1}^{\ell(\lam)}(-1)^{\lam_i-1},$ we can again factor the part for a given $W$ and $\lam$ to get $$\ol{X}_G = \sum_{W\se V}(-1)^{|V\bs W|}\sum_\lam \frac1{z_\lam}\prod_{i=1}^{\ell(\lam)} (-1)^{\lam_i-1}|\mc{P}_{G|_W}(\lam)|\sum_{d\ge 1}\mu(d)\ln(1+\ol{p}'_{d\lam_i}).$$ By the same reasoning as in \S \ref{sec:omega_new_ps_q}, we can rewrite this as an exponential and then reorganize the terms based on the value of $k = dj$ to get
\begin{align*}
    \ol{X}_G &= \sum_{W \se V}(-1)^{|V\bs W|}\exp\left(\sum_{j\ge 1}\frac{(-1)^{j-1}}{j}|\mc{P}_{G|_W}(j)|\sum_{d\ge 1}\frac{\mu(d)}{d}\ln(1+\ol{p}'_{dj})\right) \\
    &= \sum_{W \se V}(-1)^{|V\bs W|}\exp\left(\sum_{k\ge 1}\ln(1+\ol{p}'_{k})\cdot\frac1k\sum_{d\ge 1} (-1)^{k/d-1}\frac{\mu(d)}{d}\left|\mc{P}_{G|_W}\left(\frac kd\right)\right|\right).
\end{align*}
This is the same M\"obius inversion formula as in (\ref{eqn:c_mobius}) from Claim \ref{claim:c}, so it follows that $$\ol{X}_G = \sum_{W\se V}(-1)^{|V\bs W|}\exp\left(\sum_{k\ge 1}\ln(1+\ol{p}'_k)\cdot c_W(k)\right) = \sum_{W \se V}(-1)^{|V\bs W|}\prod_{k\ge 1}(1+\ol{p}_k')^{c_W(k)},$$ recovering the factorization used in \S \ref{sec:XG_new_ps} and \S\ref{sec:XG_new_ps_2} to prove Theorem \ref{thm:XG_new_ps}.

\subsection{\texorpdfstring{$\ol{p}$}{p}-expansion for \texorpdfstring{$\omega(\ol{X}_G(q))$}{omega(XG(q))}}\label{sec:omega_old_ps_q}

We use a similar strategy to \S \ref{sec:omega_new_ps_q} to find the transition formula from the $p$-basis to the $\ol{p}$-basis. By (\ref{eqn:H_E_factors}), 
\begin{equation}\label{eqn:p_E}
    1+\ol{p}_k = \prod_{i\ge 1}(1+x_i^k) = \sum_{n\ge 1}e_n[p_k] = E(1)[p_k].
\end{equation} 
Recalling from \S \ref{sec:background} that $\ln(E(t)) = -P(-t)$, taking logarithms on both sides of (\ref{eqn:p_E}) gives 
\begin{equation}\label{eqn:p_transition_2}
    \ln(1+\ol{p}_k) = -P(-1)[p_k] = \sum_{n\ge 1} \frac{(-1)^{n+1}p_{n}[p_k]}{n} = \sum_{n\ge 1} \frac{(-1)^{n+1}p_{nk}}{n}.
\end{equation} Now we can isolate $p_k$ using the Dirichlet inverse $\hat{\mu}$ of the sequence $(-1)^{n+1}$ from Lemma \ref{lem:dirichlet}: $$p_k = \sum_{n\ge 1}\frac{p_{nk}}{n}\sum_{d\mid n}(-1)^{n/d+1}\hat{\mu}(d) = \sum_{d\ge 1}\frac{\hat{\mu}(d)}{d}\sum_{j\ge 1}\frac{(-1)^{j+1}p_{jdk}}{j} = \sum_{n\ge 1}\frac{\hat{\mu}(d)}{d}\ln(1+\ol{p}_{dk}),$$ using the substitution $j=n/d$ as before. Expanding the logarithm power series and then multiplying over the parts of an arbitrary partition $\lam$ gives 
\begin{equation}\label{eqn:old_ps}
    p_\lam = \sum_{d,n}\prod_{i=1}^{\ell(\lam)}\frac{\hat{\mu}(d_i)(-1)^{n_i+1}}{d_in_i}\ol{p}_{(d\lam)^n},
\end{equation} 
where again $d = (d_1,\dots,d_{\ell(\lam)})$ and $n = (n_1,\dots,n_{\ell(\lam)})$ are compositions with all parts positive. Plugging (\ref{eqn:p_transition_2}) and (\ref{eqn:old_ps}) into (\ref{eqn:XG_q_final}) gives:

\begin{prop}\label{prop:omega_old_ps_q}
    If $G$ is the incomparability graph of a unit interval order, $$\omega(\ol{X}_G(q)) = \sum_\lam \frac1{z_\lam}\sum_{P \in \mc{P}^{\tn{all }V}_G(\lam)} q^{\tn{asc}(P)}\prod_{i=1}^{\ell(\lam)}\sum_{d\ge 1}\frac{\hat{\mu}(d)}{d}\ln(1+\ol{p}_{d\lam_i}),$$ and the $\ol{p}$-coefficients for $\omega(\ol{X}_G(q)$ are $$\displaystyle{[\ol{p}_\lam]\omega(\ol{X}_G(q)) = \sum_{(d\lam')^n=\lam}\frac{1}{z_{\lam'}}\sum_{P\in \mc{P}_G(\lam')} q^{\tn{asc}(P)}}\prod_{i=1}^{\ell(\lam')}\frac{\hat{\mu}(d_i)(-1)^{n_i+1}}{d_i n_i},$$ where $\hat{\mu}(n) = \begin{cases}
            \mu(n) & \tn{if }n\tn{ is odd}, \\
            2^{j-1}\mu(n/2^j) &\tn{if }n\tn{ is even with }n/2^j\tn{ odd}.
        \end{cases}$
\end{prop}

We can recover the factorization for Theorem \ref{thm:omega_XG_old_ps} from Proposition \ref{prop:omega_old_ps_q} using the same approach as in \S \ref{sec:omega_new_ps_q} and \S \ref{sec:new_ps_q}. We can take an alternating sum over $W \se V$ and factor over the parts of $\lam$ to get $$\omega(\ol{X}_G) = \sum_{W \se V}(-1)^{|V\bs W|}\frac1{z_\lam}\prod_{i=1}^{\ell(\lam)}|\mc{P}_{G|_W}(\lam_i)|\sum_{d\ge 1}\frac{\hat{\mu}(d)}{d}\ln(1+\ol{p}_{d\lam_i}),$$ and then we can interpret this as an exponential and reorganize the terms to get
\begin{align*}
    \omega(\ol{X}_G) &= \sum_{W \se V}(-1)^{|V\bs W|}\exp\left(\sum_{j\ge 1}\frac1j|\mc{P}_{G|_W}(j)|\sum_{d\ge 1}\frac{\hat{\mu}(d)}{d}\ln(1+\ol{p}_{dj})\right) \\
    &= \sum_{W \se V}(-1)^{|V\bs W|}\exp\left(\sum_{k\ge 1}\ln(1+\ol{p}_{k})\cdot\frac1k\sum_{d\ge 1} \hat{\mu}(d)\left|\mc{P}_{G|_W}\left(\frac kd\right)\right|\right).
\end{align*}
Since this is the same $\hat{\mu}$ from Lemma \ref{lem:dirichlet}, by the same M\"obius inversion argument as in \S \ref{sec:omega_XG_old_ps}, this becomes $$\omega(\ol{X}_G) = \sum_{W\se V}(-1)^{|V\bs W|}\exp\left(\sum_{k\ge 1}\ln(1+\ol{p}_k)\cdot b_W(k)\right) = \sum_{W\se V}(-1)^{|V\bs W|}\prod_{k\ge 1}(1+\ol{p}_k)^{b_W(k)},$$ matching the factorization from \S \ref{sec:omega_XG_old_ps} and \S \ref{sec:omega_XG_old_ps_2} that we used to prove Theorem \ref{thm:omega_XG_old_ps}.

\subsection{\texorpdfstring{$\ol{p}$}{p}-expansion for \texorpdfstring{$\ol{X}_G(q)$}{XG(q)}}\label{sec:old_ps_q}

By the similar reasoning to \S \ref{sec:new_ps_q}, we can use the fact that $\omega(p_\lam) = (-1)^{|\lam|-\ell(\lam)}p_\lam$ together with the derivation in \S \ref{sec:omega_old_ps_q} to get:

\begin{prop}\label{prop:old_ps_q}
    If $G$ is the incomparability graph of a unit interval order, $$\ol{X}_G(q) = \sum_\lam \frac{(-1)^{|\lam|-\ell(\lam)}}{z_\lam}\sum_{P \in \mc{P}^{\tn{all }V}_G(\lam)} q^{\tn{asc}(P)}\prod_{i=1}^{\ell(\lam)}\sum_{d\ge 1}\frac{\hat{\mu}(d)}{d}\ln(1+\ol{p}_{d\lam_i}),$$ and the $\ol{p}$-coefficients for $\ol{X}_G(q)$ are $$\displaystyle{[\ol{p}_\lam]\omega(\ol{X}_G(q)) = \sum_{(d\lam')^n=\lam}\frac{(-1)^{|\lam'|-\ell(\lam')}}{z_{\lam'}}\sum_{P\in \mc{P}_G(\lam')} q^{\tn{asc}(P)}}\prod_{i=1}^{\ell(\lam')}\frac{\hat{\mu}(d_i)(-1)^{n_i+1}}{d_i n_i},$$ where $\hat{\mu}(n) = \begin{cases}
            \mu(n) & \tn{if }n\tn{ is odd}, \\
            2^{j-1}\mu(n/2^j) &\tn{if }n\tn{ is even with }n/2^j\tn{ odd}.
        \end{cases}$
\end{prop}

We can relate Proposition \ref{prop:old_ps_q} to Theorem \ref{thm:XG_old_ps} by the same reasoning as before. Setting $q=1,$ taking an alternating sum over $W\se V$, and factoring over the parts of $\lam$ gives $$\ol{X}_G = \sum_{W \se V}(-1)^{|V\bs W|}\frac1{z_\lam}\prod_{i=1}^{\ell(\lam)}(-1)^{\lam_i-1}|\mc{P}_{G|_W}(\lam_i)|\sum_{d\ge 1}\frac{\hat{\mu}(d)}{d}\ln(1+\ol{p}_{d\lam_i}).$$ Then interpreting it in terms of exponentials and reorganizing the terms gives 
\begin{align*}
    \ol{X}_G &= \sum_{W \se V}(-1)^{|V\bs W|}\exp\left(\sum_{j\ge 1}\frac{(-1)^{j+1}}j|\mc{P}_{G|_W}(j)|\sum_{d\ge 1}\frac{\hat{\mu}(d)}{d}\ln(1+\ol{p}_{dj})\right) \\
    &= \sum_{W \se V}(-1)^{|V\bs W|}\exp\left(\sum_{k\ge 1}\ln(1+\ol{p}_{k})\cdot\frac1k\sum_{d\ge 1} (-1)^{k/d+1}\hat{\mu}(d)\left|\mc{P}_{G|_W}\left(\frac kd\right)\right|\right).
\end{align*}
By the M\"obius inversion argument from \S \ref{sec:XG_old_ps}, this becomes $$\ol{X}_G = \sum_{W\se V}(-1)^{|V\bs W|}\exp\left(\sum_{k\ge 1}\ln(1+\ol{p}_k)\cdot a_W(k)\right) = \sum_{W\se V}(-1)^{|V\bs W|}\prod_{k\ge 1}(1+\ol{p}_k)^{a_W(k)},$$ with the $a_W(k)$'s from \S \ref{sec:XG_old_ps}. Thus, we recover the factorization from \S \ref{sec:XG_old_ps} and \S \ref{sec:XG_old_ps_2}, and hence Theorem \ref{thm:XG_old_ps}.

\section*{Acknowledgments}

I thank Oliver Pechenik for the problem suggestion and helpful discussions related to it. I thank Boon Leong Ng for a correction to the list of graphs after Corollary \ref{cor:subgraphs}.

\section*{Conflict of interest statement}

On behalf of all authors, the corresponding author states that there is no conflict of interest.

\section*{Data availability statement}

This research has no associated data.

\printbibliography

\end{document}